\documentclass[12pt,reqno]{amsart}
\usepackage{hyperref}
\usepackage{amssymb}
\usepackage{graphicx}
\usepackage[all]{xy}
\DeclareFontFamily{U}{mathb}{\hyphenchar\font45}
\DeclareFontShape{U}{mathb}{m}{n}{
      <5> <6> <7> <8> <9> <10> gen * mathb
      <10.95> mathb10 <12> <14.4> <17.28> <20.74> <24.88> mathb12
      }{}
\DeclareSymbolFont{mathb}{U}{mathb}{m}{n}
\DeclareMathSymbol{\righttoleftarrow}{3}{mathb}{"FD}

\theoremstyle{plain}
\newtheorem{prop}{Proposition}[section]
\newtheorem{theo}[prop]{Theorem}

\newtheorem{lemm}[prop]{Lemma}
\theoremstyle{remark}
\newtheorem{rema}[prop]{Remark}
\newtheorem{ques}[prop]{Question}

\theoremstyle{definition}

\newtheorem{exam}[prop]{Example}
\numberwithin{equation}{section}

\newcommand{\actsfromleft}{\mathrel{\reflectbox{$\righttoleftarrow$}}}
\newcommand{\actsfromright}{\righttoleftarrow}

\DeclareMathOperator{\Aut}{Aut}

\DeclareMathOperator{\BirAut}{BirAut}

\DeclareMathOperator{\Burn}{Burn}

\DeclareMathOperator{\bCodim}{codim}

\DeclareMathOperator{\Pic}{Pic}

\DeclareMathOperator{\Hom}{Hom}

\DeclareMathOperator{\SL}{SL}

\DeclareMathOperator{\Ind}{Ind}

\DeclareMathOperator{\tran}{tran}

\newcommand{\ra}{\rightarrow}

\newcommand{\bC}{{\mathbb C}}
\newcommand{\bF}{{\mathbb F}}
\newcommand{\bG}{{\mathbb G}}

\newcommand{\bN}{{\mathbb N}}
\newcommand{\bP}{{\mathbb P}}
\newcommand{\bQ}{{\mathbb Q}}

\newcommand{\bZ}{{\mathbb Z}}

\newcommand{\bCA}{{\mathcal A}}
\newcommand{\bCB}{{\mathcal B}}

\newcommand{\bCO}{{\mathcal O}}

\newcommand{\cS}{{\mathcal S}}

\newcommand{\cB}{{\mathcal B}}
\newcommand{\cM}{{\mathcal M}}

\newcommand{\oA}{\overline A}

\newcommand{\oH}{\overline H}

\newcommand{\fS}{\mathfrak S}

\begin{document}
\title[Symbols and equivariant birational geometry]{Symbols and equivariant birational geometry in small dimensions}

\author{Brendan Hassett}
\address{
Brown University,
Box 1917,
151 Thayer Street,
Providence, RI 02912, USA}
\email{bhassett@math.brown.edu}

\author{Andrew Kresch}
\address{
  Institut f\"ur Mathematik,
  Universit\"at Z\"urich,
  Winterthurerstrasse 190,
  CH-8057 Z\"urich, Switzerland
}
\email{andrew.kresch@math.uzh.ch}

\author{Yuri Tschinkel}
\address{
  Courant Institute,
  251 Mercer Street,
  New York, NY 10012, USA
}

\email{tschinkel@cims.nyu.edu}
\address{Simons Foundation\\
160 Fifth Avenue\\
New York, NY 10010\\
USA}

\date{October 16, 2020}

\begin{abstract}
We discuss the equivariant Burnside group and related new invariants in equivariant birational geometry, with a special emphasis on applications in low dimensions. 
\end{abstract}

\maketitle

Let $G$ be a finite group.
Suppose that $G$ acts birationally and generically freely
on a variety over $k$, an algebraically closed field of characteristic
zero. After
resolving indeterminacy and singularities we may assume that $G$ acts
regularly on a smooth projective variety $X$.
Classifying such actions 
$$
X\actsfromright G,
$$ 
up to birational conjugation,
especially in cases where 
$X \stackrel{\sim}{\dashrightarrow} \bP^n$, is a long-standing
problem. This entails understanding realizations of $G$
in the {\em Cremona group} $\BirAut(\bP^n)$. 
There is an enormous literature on this subject;
we will summarize some key results in Section~\ref{sect:history}.

The focus of this survey is a new approach to the analysis of 
actions via new invariants 
extracted from the fixed points and stabilizer loci of $X$.
This new approach has its origin in the study of specializations
of birational type. Suppose a smooth projective variety specializes to
a reduced normal crossing divisor -- we seek to gain information
about the general fiber from combinatorial structures associated with
the special fiber \cite{NicSh,KT}. That circle of ideas, in turn, 
was inspired by motivic integration -- another
instance of this philosophy \cite{LooijengaSurvey}. Actions of finite
groups yield formally similar stratifications: we have the open subset, on which
the group acts freely, and the locus with nontrivial stabilizers. 
The combinatorics
of the resulting stratification -- along with the representations of the
stabilizers of the strata on the normal directions -- sheds light on the original
group action. In many cases, we can extract additional information from the
birational type of these strata and the action of the normalizer of
the stabilizer. 
We explain these constructions in detail, with many examples
of small dimensions, referring to the original papers 
\cite{kontsevichpestuntschinkel,kreschtschinkel}. 
Our goal is to illustrate how the layers of this new formalism reveal various levels of structure
among $G$-birational types. 
In particular, we apply these new obstructions to cyclic actions on cubic fourfolds, including rational ones, 
and produce examples of nonlinearizable actions.

Here is the roadmap of the paper: After briefly recalling several classical results, we discuss the case
of finite {\em abelian} groups $G$. 
We introduce the group of {\em equivariant birational types} 
$\mathcal B_n(G)$, as a quotient of a $\bZ$-module of certain symbols by explicit defining relations
and find a simpler presentation of these relations (Proposition~\ref{prop:81}).  
In Section~\ref{sect:surf} we explain, in numerous examples, how to compute the invariants on surfaces. 
In Section~\ref{sect:recon} we show that for $G=C_p$, of prime order $p$, 
{\em all} symbols in  $\mathcal B_n(G)$ are represented by 
smooth projective varieties with $G$-action. 
In Section~\ref{sect:refined} we introduce and study 
refined invariants of {\em abelian} actions, taking into account not only the 
representation on the tangent bundle to the fixed point strata, but also birational types of these strata. 
In Section~\ref{sect:cubic4} we exhibit cyclic actions on cubic fourfolds that 
are not equivariantly birational to linear actions; our main goal is to highlight the applications of 
the different invariants in representative examples.  
Finally, in Section~\ref{sect:nonab} we consider {\em nonabelian} groups, define the {\em equivariant Burnside group},
which encodes new obstructions to equivariant rationality, and show how these obstructions work in a striking example, 
due to Iskovskikh \cite{isk-s3}:  distinguishing two birational 
actions of $\mathfrak S_3\times C_2$ on rational surfaces.

\medskip
\noindent
\textbf{Acknowledgments:}
The first author was partially supported by NSF grant 1701659 and Simons Foundation award 546235.
The second author was partially supported by the
Swiss National Science Foundation. The third author was partially supported by NSF grant 2000099.

\section{Brief history of previous work}
\label{sect:history}
These questions were considered by Bertini, Geiser, 
De Jonqui\`eres, Kantor, etc. over a century ago but continue to inspire
new work:
\begin{itemize}
\item{Manin \cite{manin1}, \cite{manin2} studied $G$-surfaces both in the arithmetic and geometric 
context, focusing on the induced $G$-action on the geometric Picard group, and on 
cohomological invariants of that lattice;}
\item{Iskovskikh \cite{IskMin} laid the groundwork for the $G$-birational
classification of surfaces and their linkings; 
}
\item{Bayle, Beauville, Blanc, and de Fernex \cite{baylebeauville,BeaBla,deFer,BlaGGD,blancsubgroups}
classified actions of finite abelian $G$ on surfaces;}
\item{Dolgachev and Iskovskikh \cite{DolIsk} largely completed the surface case;}
\item{Bogomolov and Prokhorov \cite{BogPro,prokhorovII} considered the stable conjugacy problem
for the surface case using cohomological tools introduced by Manin;}
\item{Prokhorov, Cheltsov, Shramov, and collaborators proved numerous theorems for
threefolds -- both concerning specific groups, such as simple groups \cite{ProSimple,CheShr},
as well as general structural properties \cite{PSJordan}.}
\end{itemize}
Much of this work fits into the Minimal Model Program,
using distinguished models to reduce the classification problem 
to an analysis of automorphisms of a restricted class of 
objects, e.g., del Pezzo surfaces.
With a few exceptions -- the application of the cohomology on the
N\'eron-Severi group, by Manin and Bogomolov--Prokhorov, and the `normalized
fixed curve with action (NFCA)' invariant of Blanc \cite{blancsubgroups} 
-- invariants play a limited role.

\

A fundamental observation, recorded in \cite[App.\ A]{reichsteinyoussinessential},
is that the presence of a
point fixed by a given abelian subgroup $H$ of $G$ is a birational invariant of
a smooth projective variety $X$ with generically free $G$-action.
Furthermore, Reichstein and Youssin  showed that 
the {\em determinant} of the action of abelian stabilizers on the
normal bundle to the locus with the given stabilizer, up to sign, is also a birational invariant 
\cite{reichsteinyoussininvariant}. However, for finite groups this is only 
useful when the minimal number of
generators of the abelian group equals the dimension of the variety 
\cite[Th.~1.1]{reichsteinyoussinessential}. For cyclic groups, it is
applicable only for curves.

The invariants defined in \cite{kontsevichpestuntschinkel,Bbar,kreschtschinkel} 
record all eigenvalues for the action 
of abelian stabilizers, as well as additional information about the action on the
components of the fixed locus, and on their
function fields. These collections of 
data are turned into a $G$-birational invariant, via explicit {\em blowup relations}. The groups
receiving these invariants, the {\em equivariant Burnside groups}, have an elaborate algebraic structure. 
And they led to new results in birational geometry, some of which will be discussed below.

\section{Equivariant birational types}
\label{sect:first}

Here we restrict to the situation where $G$ is {\em abelian}
and consider only fixed points of
$X \actsfromright G$. In general, there are no such 
fixed points and we obtain no information. However, large classes of actions
do have fixed points, e.g., if $G$ is cyclic and 
$h^i(X,\bCO_X)=0$, for each $i>0$, then the Atiyah-Bott holomorphic 
Lefschetz formula \cite[Cor.~4.13]{AtiBot} yields a fixed point.
The vanishing assumption holds for rational and rationally connected $X$.
If $G$ is an abelian $p$-group ($p$ prime) acting on $X$ without
fixed points
then every Chern number of $X$ is divisible by $p$ \cite[Cor.~1.1.2]{Hau}.




To define an invariant of $X \actsfromright G$, we consider collections of weights for the action of $G$ 
in the tangent bundle at $G$-fixed points in $X$. To formalize this, 
let 
$$
A=G^\vee = \Hom(G,\bG_m)
$$
be the character group of $G$, and $n=\dim X$. Let 
$$
\cS_n(G)
$$ 
be the free abelian group on
symbols 
$$
[a_1,\ldots,a_n], \quad a_j\in A, \quad \forall j,
$$ 
subject to the conditions:

\

\begin{itemize}
\item[({\bf G})] {\bf Generation:}
$\{a_1,\ldots,a_n\}$ generate $A$, i.e., 
$$
\sum_{i=1}^n \bZ a_i = A,
$$
thus, $n$ is at least the minimal number of generators of $G$;

\

\item[({\bf S})]  {\bf Symmetry:} for each permutation $\sigma \in \fS_n$ we have
$$
[a_{\sigma(1)},\ldots,a_{\sigma(n)}] = [a_1,\ldots,a_n].
$$
\end{itemize}

\

\noindent
Let 
 \begin{equation}
 \label{eqn:Bn}
\cS_n(G) \to \bCB_n(G)
 \end{equation}
 be the quotient, by relations, for all $2\le r \le n$:

\

\begin{itemize}
\item[($\mathrm{\bf B}_r$)]  {\bf Blow-up:} 
for all 
$
[a_1,\ldots,a_r,b_1,\ldots,b_{n-r}]\in \cS_n(G)
$
one has 
\begin{align} 
\label{keyrelation0}
  [a_1,\ldots,a_r,b_1,\ldots,b_{n-r}] &= \nonumber  \\
\sum_{1\le i \le r, \, a_i\neq a_{i'} \text{ for }i'<i}
  [a_1-a_i,\ldots,a_i,\ldots,a_r-a_i,& b_1,\ldots,b_{n-r}].
\end{align}
\end{itemize}

\

\noindent
These relations reflect the transformations of weights in tangent spaces to components of the fixed locus 
upon blowing up along a $G$-stable stratum. 

\

\noindent
From the definition, we have
$$\bCB_1(G) = \begin{cases} \bZ^{\phi(N)}& \text{if } G \text{ is cyclic of order $N$, }  \\
                                 0    & \text{otherwise.} 
                   \end{cases}
                   $$

\begin{prop} \cite[Prop.~8.1]{kreschtschinkel}
\label{prop:81}
For $n\ge 2$, all relations $\mathrm{({\bf B}_{\it r})}$ are implied by relation  $\mathrm{({\bf B}_2)}$.
\end{prop}

\noindent
Thus,  $\bCB_n(G)$ is obtained by imposing the relation:

\

\begin{itemize}
\item[($\mathrm{\bf B}$)]  {\bf Blow-up:} 
for all 
$
[a_1,a_2,b_1,\ldots,b_{n-2}]\in \cS_n(G)
$
one has 
$$
[a_1,a_2,b_1,\ldots,b_{n-2}]  = 
$$
\begin{align}
\label{keyrelation}
& [a_1,a_2-a_1,b_1,\ldots,b_{n-2}] +  [a_1-a_2,a_2,b_1,\ldots,b_{n-2}],  & a_1\neq a_2, \\
& [0,a_1,b_1,\ldots,b_{n-2}], & a_1=a_2.  \nonumber 
\end{align}
\end{itemize}

\

\begin{proof}[Proof of Proposition~\ref{prop:81}]
We prove the result by induction on $r$.
We first treat the case that $a_1,a_2,\dots,a_r$ are
pairwise distinct; we drop the entries $b_1, \dots, b_{n-r}$ from the notation, as they do not 
take part in relations. 
Suppose $r\ge 3$.
Then:
\begin{align*}
[a_1,\dots,a_r]&=\text{(by $(\mathbf{B}_{r-1})$)}\,
                 [a_1,a_2-a_1,\dots,a_{r-1}-a_1,a_r]+\cdots \\ & \qquad\qquad\qquad\quad+
                     [a_1-a_{r-1},\dots,a_{r-2}-a_{r-1},a_{r-1},a_r] \\
                 &=\text{(by $(\mathbf{S})$)}\,\,\,\,\,\,\,\,\,\,
                 [a_1,a_r,a_2-a_1,\dots,a_{r-1}-a_1]+\cdots \\ & \qquad\qquad\qquad\quad+
                     [a_{r-1},a_r,a_1-a_{r-1},\dots,a_{r-2}-a_{r-1}] \\
                 &=\text{(by $(\mathbf{B}_2)+(\mathbf{S})$, applied to each term)} \\
&\!\!\!\!\!\!\!\!\!\!\!\!\!\!\![a_1,a_2-a_1,\dots,a_r-a_1]
                 +[a_1-a_r,a_2-a_1,\dots,a_{r-1}-a_1,a_r]\\
                 &\qquad+\,\,\cdots\,\,+\\
&\!\!\!\!\!\!\!\!\!\!\!\!\!\!\![a_1-a_{r-1},\dots,a_{r-2}-a_{r-1},a_{r-1},a_r-a_{r-1}]\\ &\qquad\qquad+[a_1-a_{r-1},\dots,a_{r-2}-a_{r-1},a_{r-1}-a_r,a_r].
\end{align*}
The right-hand terms, taken together, are equal by $(\mathbf{B}_{r-1})$ to
\[ [a_1-a_r,\dots,a_{r-1}-a_r,a_r], \]
which together with the left-hand terms gives us what we want.

Next we treat the more general case $a_r\notin \{a_1,\dots,a_{r-1}\}$.
In that case, for every $i$ with $a_i\in \{a_1,\dots,a_{i-1}\}$ we
omit the $i$th term on the right-hand side in the initial application
of $(\mathbf{B}_{r-1})$ and $(\mathbf{S})$ and omit the $i$th line
after the applications of $(\mathbf{B}_2)$ and $(\mathbf{S})$.
We conclude as before.

Finally we treat the case $a_r\in \{a_1,\dots,a_{r-1}\}$.
We start in the same way, by applying $(\mathbf{B}_{r-1})$ and $(\mathbf{S})$
as above.
Now, when we apply $(\mathbf{B}_2)$, we have to pay special attention to
terms with $a_i=a_r$: in the corresponding line we should leave the
left-hand term but omit the right-hand term.
Each of the remaining right-hand terms vanishes by
an application of $(\mathbf{B}_2)$ to a symbol of the form $[0,\dots]$,
i.e., the vanishing of any symbol with two nonzero weights summing to $0$.
The left-hand terms give us directly what we want.
\end{proof}

\subsection{Antisymmetry} 
\label{subsect:anti}

Let 
$$
\bCB_{n}(G) \to \bCB^{-}_{n}(G)
$$
be the projection to the quotient by the additional relation
\begin{equation}
\label{eqn:anti}
[-a_1,\ldots, a_n]=-[a_1,\ldots, a_n].
\end{equation}
In particular, symbols of the type $[0, a_2, \ldots, a_n]$ 
are in the kernel of the projection. We write 
$$
[a_1,\ldots, a_n]^-
$$
for the image of a standard generator $[a_1,\ldots, a_n]$.

\

\subsection{Multiplication and co-multiplication}
\label{subsect:mult}

Consider a short exact sequence 
$$
0\ra G'\ra G\ra G''\ra 0
$$
and its dual
$$
0\ra A''\ra A\ra A'\ra 0.
$$
The {\em multiplication} 
$$
\nabla: \bCB_{n'}(G')\otimes \bCB_{n''}(G'')\ra \bCB_{n}(G), \quad n'+n''=n, 
$$
is defined by
$$
[a_1',\ldots, a_{n'}']\otimes [a_1'',\ldots, a_{n''}''] \mapsto \sum \,\, [a_1,\ldots, a_{n'}, a_1'', \ldots a_{n''}''],
$$
summing over all lifts $a_i \in A$ of $a_i'\in A'$. It descends to a similar map on quotients by 
the relation \eqref{eqn:anti}. . 

The {\em co-multiplication} is defined only on $\bCB^{-}_{n}(G)$: 
$$
\Delta^-: \bCB^{-}_{n}(G) \ra \bCB_{n'}^-(G')\otimes \bCB_{n''}^-(G''), \quad n'+n''=n.
$$ 
On generators it takes the form
$$
[a_1,\ldots, a_n]^-\mapsto \sum \,\, [a_{I'} \, \mathrm{ mod }\,\,  A'']^-\otimes [a_{I''}]^-,
$$
where the sum is over subdivisions $\{1,\ldots, n\} = I'\sqcup I''$ of cardinality $n'$, respectively $n''$, 
such that 
\begin{itemize}
\item $a_j\in A''$, for all $j\in I''$
\item $a_j, j\in I''$, generate $A''$. 
\end{itemize}
The correctness of this definition is proved as in \cite[Prop.~11]{kontsevichpestuntschinkel}. Here are the main steps: By \cite[Prop.~8.1]{kreschtschinkel} (= Proposition~\ref{prop:81}), it suffices to check 2-term relations
$\mathrm{({\bf B}_2)}$, i.e., the image of the relation
$$
[a_1,a_2, \ldots  ] ^-= [a_1-a_2, a_2,\ldots]^- + [a_1,a_2-a_1, \ldots]^-
$$
after applying co-multiplication. 
The only interesting part is when the first two arguments are distributed 
over different factors in the definition 
of co-multiplication. 

The relation is the same relation as that for the $\mathcal M_n(G)$-groups, introduced and studied in \cite{kontsevichpestuntschinkel},
unless, $a_1=a_2$ -- recall that
$$
[a,a, \ldots]= [a,0,\ldots, ] \in \bCB_n(G).
$$
Since 
$$
0 = [a,-a,\ldots]^- = - [a,a,\ldots]^-
$$
we have
$$
[a,a,\ldots]^-= [a,0,\ldots]^- = 0.
$$
Now it suffices to repeat the argument in 
\cite[Prop.~11]{kontsevichpestuntschinkel}. Using the terminology of that paper, there are four cases, 
of which only (1) and (4) are relevant. In both cases, all terms are zero.

\

\subsection{Birational invariant}

We return to the definition of an invariant for $X \actsfromright G$,
$\dim(X)=n$, and $G$ abelian.
Consider irreducible components of the fixed locus
$$
X^G=\coprod_{\alpha} F_{\alpha},
$$ 
and write 
$$
\beta_{\alpha}=[a_{1,\alpha},\ldots,a_{n,\alpha}]
$$ 
for the unordered $n$-tuple of 
weights for the $G$-action on the tangent space $\mathcal T_{x_{\alpha}}X$, for some 
$x_{\alpha}\in F_{\alpha}$ -- this does not depend on the choice of $x_{\alpha}$. 
The number of zero weights is $\dim(F_{\alpha})$. 
We express 
\begin{equation}
\label{eqn:def-inv}\beta(X \actsfromright G )=\sum_{\alpha} \, \beta_{\alpha}
\in \bCB_n(G),
\end{equation}
and write
$$
\beta^-(X\actsfromright G) \in \bCB^{-}_{n}(G),
$$
for the image under the projection.

\begin{theo} \cite[Th.~3]{kontsevichpestuntschinkel}
\label{thm:inv}
The class 
$$
\beta(X \actsfromright G ) \in \bCB_n(G)
$$
is a $G$-birational invariant.
\end{theo}

The proof relies on $G$-equivariant Weak Factorization, connecting $G$-birational varieties 
via blow-ups and blow-downs of smooth $G$-stable subvarieties.

\begin{prop}
\label{prop:cn}
Consider a linear, generically free, action of a cyclic group $C_N$, of order $N$, on $\bP^n$, for $n\ge 2$. 
Then 
$$
\beta^-(\bP^n\actsfromright C_N)=0 \in \bCB^{-}_{n}(C_N).
$$
\end{prop}

\begin{proof}
We know that all such actions are equivariantly birational, see, e.g., \cite[Th. 7.1]{reichsteinyoussininvariant}. 
Thus it suffices to consider one such action. Take an action with weights $(1,0,\ldots, 0)$. It fixes a hyperplane, and a point, the corresponding class is
$$
[1,0,\ldots, 0] + [-1, -1, \ldots, -1]=[1,0,\ldots, 0] + [-1,0,\ldots, 0],    
$$
here, we repeatedly used relation \eqref{keyrelation} to transform the second term. 
\end{proof}

\begin{rema} \label{rema:lineartorsion}
We shall show in Section~\ref{subsect:torsion} that
$[1,0]+[-1,0]$ is torsion in $\bCB_2(C_N)$.
See \cite[Prop.~7 and Lem.~32]{kontsevichpestuntschinkel} for the case 
of prime $N$.
The element is nontrivial when
$$
N=7, 9, 10, 11, 13, 14, 15, 17, \ldots.
$$
\end{rema}

\section{Computation of invariants on surfaces}
\label{sect:surf}

\subsection{Sample computations of $\bCB_2(C_p)$}
This group is generated by symbols $[a_1,a_2]$, where
$$
a_1,a_2 \in \Hom(C_p,\bG_m)\cong \bZ/p\bZ
$$ 
are not both trivial.
To simplify notation, we write $a_i=0,1,\ldots,p-1$.
Note that 
$$
[a,a]=[0,a]\quad \text{ and } \quad [a_1,a_2]=[a_2,a_1]
$$
so that the 
symbols 
$$[a_1,a_2]: \quad \quad 0<a_1\le a_2 < p$$
generate $\bCB_2(C_p)$.
The other relation -- for $a_1<a_2$ -- takes the form
$$
[a_1,a_2]=[a_1,a_2-a_1]+[p+a_1-a_2,a_2].
$$

\

\begin{itemize}
\item[($p=2$)]
We obtain relations
$$[1,1]=[0,1], \quad  [0,1]=[0,1]+[1,1]$$
forcing $\bCB_2(C_2)=0$.
\end{itemize}

\

\begin{itemize}
\item[($p=3$)] 
We obtain relations
\begin{align*}
[0,1] &= [0,1] + [2,1] = [0,1] + [1,2] \\
[0,2] &= [0,2] + [1,2] \\
[1,1] &= [0,1] \\
[1,2] &= [1,1] + [2,2] \\
[2,2] &= [0,2]
\end{align*}
forcing 
$$
[1,2]=0\quad \text{  and } \quad [0,1]=[1,1]=-[2,2]=-[0,2].
$$ 
We conclude that $\bCB_2(C_3)\cong \bZ$.
\end{itemize}

For instance, consider the standard diagonal action on $\bP^2$
$$ (x:y:z) \mapsto (x:\zeta_3 y:\zeta_3^2 z)$$
with the convention that $\zeta_N$ is a primitive $N$th root of unity.
This has invariant 
$$\beta(\bP^2 \actsfromright C_3) = [1,2]+[1,2]+[1,2]=0.$$

On the other hand, the action of $C_3$ on the cubic surface 
$$X=\{w^3=f_3(x,y,z)\}, \quad (w:x:y:z)\mapsto (\zeta_3 w:x:y:z)$$
fixes the cubic curve $\{ w=0\}$ and we find that 
$$\beta(X\actsfromright C_3)=[2,0]\neq 0,$$
thus $X$ is not $G$-equivariantly birational to $\bP^2$, with linear action. 
Note that $\beta$ does not allow to distinguish among these cubic
surfaces.
Nor does it distinguish the cubic surfaces from the
degree one del Pezzo surfaces with $C_3$-action
$$Y=\{w^2=z^3 + f_6(x,y) \} \subset \bP(3,1,1,2),
$$
given by
$$
(w:x:y:z)\mapsto (w:x:y:\zeta_3 z).
$$
We shall see in Section~\ref{sect:refined} that
taking into account the fixed locus
$[F_{\alpha}]$ gives a complete invariant. 

\

\begin{itemize}
\item[($p=5$)]
We have relations
\begin{align*}
[0,1] &= [0,1] + [4,1] \\
[0,2] &= [0,2] + [3,2] \\
[0,3] &= [0,3] + [2,3] \\
[0,4] &= [0,4] + [1,4] 
\end{align*}
forcing $[1,4]=[2,3]=0$. We also have
\begin{align*}
[1,2] &= [1,1] + [4,2] \\
[1,3] &= [1,2] + [3,3] \\
[1,4] &= [1,3] + [2,4] \\
[2,3] &= [2,1] + [4,3] \\
[2,4] &= [2,2] + [3,4] \\
[3,4] &= [3,1] + [4,4] 
\end{align*}
which shows that $\bCB_2(C_5)$ is freely generated by 
$$
\beta_1=[1,1]\quad \text{  and }\quad \beta_2=[1,2]
$$
with 
\begin{align*}
[1,3]=\beta_1-\beta_2, \quad  [2,2]=2\beta_2-\beta_1, \quad [2,4]=\beta_2-\beta_1,\\
 [3,3]=\beta_1-2\beta_2, \quad  [3,4]=-\beta_2, \quad [4,4]=-\beta_1.
\end{align*}
\end{itemize}

For example, $\overline{M}_{0,5}$, a del Pezzo surface of degree 5, has a natural action of $C_5$ by
permuting the coordinates with fixed points given by the roots
of $z^5-1$ and $z^5+1$. We compute $\beta(\overline{M}_{0,5} \actsfromright C_5)$:
$$
\beta_1 +\beta_2+(\beta_1-\beta_2)+(2\beta_2-\beta_1)+
(\beta_2-\beta_1)+(\beta_1-2\beta_2)+(-\beta_2)+(-\beta_1) =0.$$
Indeed, this action is in fact conjugate \cite{BeaBla} to a linear action on
$\bP^2$
$$ (x:y:z) \mapsto (x:\zeta_5 y:\zeta_5^4 z).$$
However, there is also a nontrivial action of $C_5$ on a del Pezzo surface of degree 1:
$$X=\{w^2=z^3+\lambda x^4z + x(\mu x^5 + y^5) \} \subset \bP(3,1,1,2),
$$
$$
(w:x:y:z) \mapsto (w:x:\zeta_5 y:z),
$$
with fixed locus an elliptic curve and invariant
$\beta(X \actsfromright C_5)=[4,0].$

\

Let us compute an example of nonprime order. The group $\bCB_2(C_4)$ has
generators 
$$[0,1],[0,3],[1,1],[1,2],[1,3],[2,3],[3,3]$$
and relations
\begin{align*}
[0,1] &= [0,1]+[3,1] \\
[0,3] &= [0,3]+[1,3] \\
[1,1] &= [0,1] \\
[1,2] &= [1,1] + [3,2] \\
[1,3] &= [1,2] + [2,3] \\
[2,3] &= [2,1] + [3,3] \\
[3,3] &= [0,3]
\end{align*}
whence 
$$[1,3]=0, \quad \beta_1:=[1,2]=-[2,3], \quad  [0,3]=[3,3]=2[2,3]=-2\beta_1, 
$$
$$
[0,1]=[1,1]=2[1,2]=2\beta_1
$$
and $\bCB_2(C_4)\cong \bZ$.

\

Consider the del Pezzo surface of degree 1, given by
$$
X=\{ w^2 = z^3 + zL_2(x^2,y^2) + xy M_2(x^2,y^2) \} \subset \bP(3,1,1,2),
$$
where $L_2$ and $M_2$ are homogeneous of degree two. It admits a $C_4$-action
$$
(w:x:y:z) \mapsto (iw:x:-y:-z)
$$
with a unique fixed point $(0:1:0:0)$. The weights on the tangent bundle
are $[2,3]$ whence
$$
\beta(X\actsfromright C_4)\neq 0.
$$
Observe that $X^{C_2}$ is a curve of genus four.

\

See \cite[\S 10.1]{blancthesis} for
a classification of automorphisms of large finite order $N$
on del Pezzo surfaces:
\begin{enumerate}
\item{The surface 
$$X=\{w^2=z^3+x(x^5+y^5) \} \subset \bP(3,1,1,2)$$
admits an automorphism of order $30$
$$(w:x:y:z) \mapsto (-w:x:\zeta_5 y: \zeta_3 z)$$
with fixed point $(0:0:1:0)$ and with weights $[3,2]$, 
thus
$$
\beta(X \actsfromright C_{30})=[3,2]\neq 0 \in \bCB_2(C_{30})\otimes \bQ,
$$ 
by a computation in {\tt Sage} \cite{sagemath}. This implies 
(see Remark~\ref{rema:lineartorsion}) that
this action is not conjugate to a linear action. 
Note that $\dim \bCB_2(C_{30})\otimes \bQ = 33.$ 
}
\item{The surface
$$X=\{w^2=z^3+xy(x^4+y^4) \} \subset \bP(3,1,1,2)$$
admits an automorphism of order $24$
$$(w:x:y:z)\mapsto (\zeta_8w : x : iy : -i\zeta_3 z).$$
The fixed point is $(0:1:0:0)$ with symbol $[21,22]$.
Computing via {\tt Sage} we find
$$\beta(X\actsfromright C_{24}) \neq 0 \in \bCB_2(C_{24})\otimes \bQ.$$
Here $\dim \bCB_2(C_{24})\otimes \bQ = 23.$
}
\end{enumerate}

\

There are good reasons why we obtain nonvanishing invariants only
when a curve is fixed: If $G=C_N$ is a cyclic group acting generically
freely
on a complex projective smooth rational surface $X\actsfromright G$ 
then the following are equivalent \cite[Th.~4]{blancthesis}:
\begin{itemize}
\item{No $g \neq 1 \in G$ fixes a curve in $X$ of positive genus.}
\item{The subgroup $G$ is conjugate to a subgroup of $\Aut(\bP^2)$.}
\end{itemize}
Even more is true: if $G$ contains an element fixing a curve of positive genus then 
$X$ is not even stably $G$-birational to projective space with a linear $G$-action, indeed, in this case
$\mathrm H^1(G,\Pic(X))\neq 0$  \cite{BogPro}.

\subsection{Examples for noncyclic groups}
If $G$ is a noncyclic abelian group then $\cB_1(G)=0$ by
definition but there are actions on curves:
\begin{exam} \label{exam:Klein1}
Consider the action of $C_2\times C_2$ on $\bP^1$ by 
$$g_1:=\left( \begin{matrix} 0 & -1 \\
		        1 & 0 \end{matrix} \right), \quad
g_2:=\left( \begin{matrix} 1 & 0 \\
    		      0 & -1 \end{matrix} \right),$$
with the elements
$$g_1^2=g_1g_2g_1^{-1}g_2^{-1} = \left( \begin{matrix} -1 & 0 \\	
					         0 & -1 \end{matrix} \right)$$
acting trivially. Thus we obtain 
$$\bP^1 \actsfromright (C_2\times C_2).$$
The group has no fixed points whence
$\beta(\bP^1 \actsfromright (C_2\times C_2))=0.$
The cyclic subgroups do have fixed points
$$(\bP^1)^{\left<g_1\right>}=\{ (1:\pm i) \}, 
(\bP^1)^{\left<g_2\right>}=\{ (1:0),(0:1) \},  
(\bP^1)^{\left<g_1g_2\right>}=\{ (1:\pm 1) \}.$$
\end{exam}
We return to this in Example~\ref{exam:Klein2}.
In Section~\ref{subsect:burndef}, we will discuss how to incorporate information from {\em all} points with nontrivial stabilizer.

We compute $\bCB_2(C_2\times C_2)$. Writing
$$(C_2 \times C_2)^{\vee} = \{0,\chi_1,\chi_2,\chi_1+\chi_2\},$$
the only admissible symbols are
$$[\chi_1,\chi_2],[\chi_1,\chi_1+\chi_2],[\chi_1+\chi_2,\chi_2]$$
with relations:
\begin{align*}
[\chi_1,\chi_2] &= [\chi_1,\chi_1+ \chi_2] + [\chi_1+\chi_2,\chi_2] \\
[\chi_1,\chi_1+\chi_2] &= [\chi_1,\chi_2] + [\chi_1+\chi_2,\chi_2] \\
[\chi_1+\chi_2,\chi_2] &= [\chi_1,\chi_1+\chi_2] + [\chi_1,\chi_2].
\end{align*}
Thus we obtain the Klein four group again
$$\bCB_2(C_2\times C_2)\cong C_2 \times C_2.$$

\

The classification of finite abelian noncyclic
actions on rational surfaces
may be found in \cite[\S 10.2]{blancthesis}. Examples of actions of
$C_2\times C_2$ on rational surfaces include:

\

\noindent
On $\bP^1 \times \bP^1$:
\begin{itemize}
\item[(1)] 
$(x,y) \mapsto (\pm x^{\pm 1},y)$, without fixed points;
\item[(2)]
$(x,y) \mapsto (\pm x, \pm y)$, 
with fixed points $(0,0),(0,\infty),(\infty,0),(\infty,\infty)$,
thus
$$
\beta(\bP^1 \times \bP^1 \actsfromright C_2 \times C_2)=4[\chi_1,\chi_2]=0;
$$
\item[(3)]
the diagonal action
$$(x,y) \mapsto (-x,-y), (x,y)\mapsto(x^{-1},y^{-1})$$
has no fixed points so the symbol sum is empty;
\end{itemize}

\

\noindent
On conic fibrations over $\bP^1$:
\begin{itemize}
\item[(4)]
$(x_1:x_2)\times (y_1:y_2) \mapsto$
$$
\begin{array}{rl}
&  (x_1:-x_2)\times (y_1:y_2), \quad \text{ respectively, } \\
& (x_1:x_2) \times (y_2(x_1-bx_2)(x_1+bx_2):y_1(x_1-ax_2)(x_1+ax_2),
\end{array}
$$
which also has four fixed points
with the same symbol;
\end{itemize}

\

\noindent
On a degree two del Pezzo surface:
\begin{itemize}
\item[(5)] 
$
(w:x:y:z) \mapsto (\pm w: x: y: \pm z)
$ on
$$
\{w^2=L_4(x,y)+z^2L_2(x,y)+z^4 \},
$$
with the involutions fixing a genus three curve and an elliptic curve
meeting in four points whence $\beta(X\actsfromright C_2\times C_2)=0$;
\end{itemize}
\

\noindent
On a degree one del Pezzo surface:
\begin{itemize}
\item[(6)]
$(w:x:y:z) \mapsto (\pm w: x: \pm y: z)$
on 
$$
\{w^2=z^3+zL_2(x^2,y^2)+L_3(x^2,y^2)\},
$$ 
with the involutions fixing a genus four curve and a genus two curve
meeting in six points whence $\beta(X\actsfromright C_2\times C_2)=0$.
\end{itemize}

None of these
actions are distinguished by $\bCB_2(C_2\times C_2)$.
Case~(1) is stably equivalent to the action on $\bP^1$ described
above. The second action is equivalent to a linear action on $\bP^2$
-- project from one of the fixed points. 
We return to these examples in
Section~\ref{subsect:examples}.

\subsection{Linear actions yield torsion classes}
\label{subsect:torsion}
Let $C_N$ act linearly and generically freely on $\bP^n$.
We saw in Proposition~\ref{prop:cn} that
$$\beta(\bP^n \actsfromright C_N) = [a,0,\ldots,0]+[-a,0,\ldots,0]$$ 
for some $a$ relatively prime to $N$.
Remark~\ref{rema:lineartorsion} pointed out this is torsion for $n\ge 2$;
we offer a proof now:
\begin{prop}
For $N\in \bN$, an $a$ with $\gcd(a,N)=1$, and $n\ge 2$ the element
$$[a,0, \ldots,0]+[-a,0,\ldots,0] \in \bCB_n(C_N)$$
is torsion.
\end{prop}

\begin{proof}
It suffices to consider $n=2$; the argument we present goes through
without changes for $n>2$. 

We will work in $\bCB_2(C_N)\otimes \bQ$. We use 
generators for this 
space arising from the alternative symbol formalism $\cM_2(C_N)$ introduced 
in \cite{kontsevichpestuntschinkel} with the property
that \cite[Prop.~7]{kontsevichpestuntschinkel}
$$
\bCB_2(C_N)\otimes \bQ = \cM_2(C_N)\otimes \bQ.
$$

For $a,b \in \Hom(C_N,\bG_m)$ generating the group we set
$$\left<a,b\right>= \begin{cases} [a,b] & \text{ if } a, b \neq 0 \\
				  \frac{1}{2}[a,0] & \text{ if } a\neq 0, b=0.
		    \end{cases}
$$
The advantage of these generators is that the relations are uniformly
$$\left<a,b\right> = \left<a,b-a\right> + \left<a-b,b\right> \quad (\mathbf{B}),$$
even when $a=b$.

We follow the proof of \cite[Th.~14]{kontsevichpestuntschinkel}. 
For all $a,b$ with $\gcd(a,b,N)=1$ we write
$$
\delta(a,b):=\left<a,b\right> + \left<-a,b\right>  + \left<a,-b\right> + \left<-a,-b\right>. 
$$
We claim this is zero in $\bCB_2(C_N)\otimes \bQ$.

First, we check that $\delta(a,b)$ is invariant under 
$\SL_2(\bZ/N\bZ)$.  
This has generators
$$
\left(\begin{matrix} 0 & -1 \\  1 & 0 \end{matrix} \right), \quad
\left(\begin{matrix} 1 & -1 \\  0 & 1 \end{matrix} \right).
$$
and $\delta(a,b)=\delta(-b,a)$ by the symmetry of the underlying symbols.
We also have
\begin{align*}
\delta(a,b-a)=& \left<a,b-a\right> + \left<-a,b-a\right> +
			 \left<a,a-b\right> + \left<-a,a-b\right> \\
	      & \text{ applying $\mathbf{B}$ to second and third terms above } \\
	     =& \left<a,b-a\right> + \left<-a,b\right> + \left<-b,b-a\right> \\
              & +\left<a,-b\right> + \left<b,a-b\right> + \left<-a,a-b\right> \\ 		
              & \text{ applying $\mathbf{B}$ to get first and four terms below} \\
     =& \left<a,b\right> + \left<-a,b\right> + \left<a,-b\right> + \left<-a,-b\right> \\
     =& \delta(a,b).
\end{align*}
		
Average $\delta(a,b)$ over 
all pairs $a,b$ with $\gcd(a,b,N)=1$ to obtain
$$S:=\sum_{a,b} \delta(a,b) = 4\sum_{a,b} \left<a,b\right>.$$
Applying the blowup
relation ($\mathbf{B}$) to all terms one finds
$$S=2S,$$
which implies that $S=0 \in \bCB_2(C_N)\otimes \bQ$.

We may regard $\delta(a,b)$ and $S$ as elements of $\bCB_2(C_N)$.
It follows that $\delta(a,b)$ is torsion in $\bCB_2(C_N)$,
annihilated by the number of summands in $S$.
Substituting $b=0$, we obtain that 
$$
\delta(a,0)=[a,0]+[-a,0]=0 \in \bCB_2(C_N)\otimes \bQ.
$$
\end{proof}
The invariance of $\delta(a,0)$ shows that $[a,0]+[-a,0]$ 
is independent of the 
choice of $a$ relatively prime to $N$.

\subsection{Algebraic structure in dimension 2}
\label{subsect:str-2}

For reference, we tabulate
$$
\dim \bCB_2(G)\otimes \bQ,
$$ 
for $G=C_N$ and small values of $N$:
$$
\begin{tabular}{r|ccccccccccccccc}
{\it N}           & 2 & 3 & 4 & 5 & 6 & 7 & 8 & 9 & 10 & 11 & 12 &13 &14 & 15 & 16 \\
\hline 
& 0 & 1 & 1 & 2 & 2 & 3 & 3 & 5 & 4 & 6 & 7 & 8 & 7 & 13 & 10 \\
\end{tabular}
$$
For primes $p\ge 5$  there is a closed formula
\cite[\S 11]{kontsevichpestuntschinkel}:
\begin{equation}
\label{label:dimb}
\dim \bCB_2(C_p) \otimes \bQ = \frac{p^2 -1}{24} + 1 = \frac{(p-5)(p-7)}{24} + \frac{p-1}{2},
\end{equation}
which strongly suggested a connection with the modular curve $X_1(p)$! 
We also have
\begin{equation}
\label{label:dimbm}
\dim \bCB_2^-(C_p) \otimes \bQ = \frac{(p-5)(p-7)}{24},
\end{equation}
and, by \cite[Prop.~30]{kontsevichpestuntschinkel}, 
$$
\bCB_1^-(C_p)\otimes \bQ= \mathrm{Ker}(\bCB_2(C_p) \to \bCB_2^-(C_p))\otimes \bQ. 
$$
Computations in noncyclic cases have been performed by Zhijia Zhang\footnote{see \url{https://cims.nyu.edu/~zz1753/ebgms/ebgms.html}}; we summarize the results: for primes $p\ge 5$ one has
$$
\dim \bCB_2(C_p\times C_p) \otimes \bQ    = \frac{(p-1)(p^3+6p^2-p+6)}{24},
$$
$$
\dim \bCB_2^-(C_p\times C_p) \otimes \bQ =\frac{(p-1)(p^3-p+12)}{24}
$$
For $G=C_{N_1}\times C_{N_2}$ and small values of $N_1,N_2$, we have:

$$
\begin{tabular}{r|ccccccccccccccc}
${\it N}_1$           & 2 & 2 & 2 & 2  & 2 & 2  & 3 & 3& 3   & 3 & 4 & 4 & 4            & 5 & 6 \\
\hline
${\it N}_2$           & 2 & 4 & 6 & 8  & 10 & 16 & 3 & 6 & 9  & 27 & 8 & 16 & 32      & 25 & 36  \\
\hline 
$d_{\bQ}$            & 0 & 2 & 3 & 6  & 7 & 21 &  7 & 15   & 37 &  235   & 33 & 105 &  353     &702  &  577  \\
$d_{\bQ}^-$         & 0 & 0 & 0  & 1  & 1 &   9  & 3  &  7 & 19 &  163    & 17 &    65     &257&  502   & 433    \\
$d_{2}$                & 2 & 5 & 8 &  13 & 18 & 36  & 7& 15 & 37 &   235    & 34 &  106 &354  & 702  & 578  \\
$d_{2}^-$             & 2 & 3 & 5 & 8   & 12 & 24  &3 & 7 & 19 &  163    & 17 & 65 & 257 &   502 &  433\\
\end{tabular}
$$

\section{Reconstruction theorem}
\label{sect:recon}

The examples offered so far might suggest that very few invariants
in $\bCB_n(G)$ are actually realized geometrically by smooth projective
varieties $X \actsfromright G$. If one allows nonrational examples far more
invariants arise:
\begin{prop}
Let $p$ be a prime. Then $\bCB_n(C_p)$ is generated as an abelian group
by $\beta(X \actsfromright C_p)$, where $X$ is smooth and projective.
\end{prop}
\begin{proof}
We proceed by induction on $n$. The case of $n=1$ follows from the 
Riemann existence theorem applied to cyclic branched covers of degree $p$ with
the prescribed ramification data.  (See also Lemma~\ref{lemm:CI} below.)

For the symbols $[a_1,\ldots,a_{n-1},0]$ we construct $(n-1)$-dimensional varieties $D$ 
with the prescribed invariants and $D\times \bP^1$ with trivial action on the second 
factor. Since $[a,a,a_3,\ldots,a_n]=[0,a,a_3,\ldots,a_n]$ we may focus on
symbols $[a_1,a_2,\ldots,a_n], 0<a_1 < a_2< \ldots <a_n<p.$ 
We are reduced to the following lemma:

\begin{lemm} \label{lemm:CI}
Any sum 
\begin{equation} \label{symbolsum}
\sum m_{[a_1,a_2,\ldots,a_n]}[a_1,a_2,\ldots,a_n]
\end{equation}
of symbols 
$$
[a_1,a_2,\ldots,a_n], \quad 0<a_1 < a_2< \ldots <a_n<p,
$$
with nonnegative coefficients,
is realized as 
$\beta(X \actsfromright C_p)$, where $X$ is smooth, projective, and irreducible.
\end{lemm}
For each symbol $[a_1,a_2,\ldots,a_n]$ appearing in the sum,
take an $n$-dimensional representation $V_{[a_1,a_2,\ldots,a_n]}$ with the prescribed
weights and the direct sum
$$W_{[a_1,a_2,\ldots,a_n]} = (V_{[a_1,a_2,\ldots,a_n]} \oplus \bC)^{m_{[a_1,a_2,\ldots,a_n]}}$$
where $\bC$ is the trivial representation of $C_p$. Write 
$$W= \oplus W_{[a_1,a_2,\ldots,a_n]},$$ 
where the index is over the terms appearing in (\ref{symbolsum}),
and consider the projectivization $\bP(W)$ and the $n$-planes
$$P_{[a_1,a_2,\ldots,a_n],j} \subset  \bP(W_{[a_1,a_2,\ldots,a_n]}), \quad j=1,\ldots,m_{[a_1,a_2,\ldots,a_n]}$$
associated with the summands of $W_{[a_1,a_2,\ldots,a_n]}$,
each with distinguished fixed point 
$$
p_{[a_1,a_2,\ldots,a_n],j}=(0:0:\cdots:0:1).
$$ 
The action on 
$$
\mathcal T_{p_{[a_1,a_2,\ldots,a_n]}}P_{[a_1,a_2,\ldots,a_n],j}
$$ 
coincides
with the action on $V_{[a_1,a_2,\ldots,a_n]}$.  The fixed points of $\bP(W)$ 
correspond to the eigenspaces for the $C_p$ action. Write 
$M=\sum m_{[a_1,a_2,\ldots,a_n]}$ for the number of summands; each
weight occurs at most $M$ times. Thus the fixed point loci are projective 
spaces of dimension $\le M-1$.

Choose a high-degree smooth complete intersection $X$ of dimension $n$, 
invariant under the action of $C_p$, containing the $p_{[a_1,a_2,\ldots,a_n],j}$ 
and tangent to $P_{[a_1,a_2,\ldots,a_n],j}$ for each $[a_1,a_2,\ldots,a_n]$.
This complete intersection exists by
polynomial interpolation applied to the quotient $\bP(W)/C_p$; smoothness
follows from Bertini's Theorem. Since complete intersections of 
positive dimension are connected, the resulting $X$ is irreducible.

It only remains to show that such a complete intersection need not have fixed
points beyond those specified. Now $X$ has codimension 
$(M-1)(n+1)$, so we may assume it avoids the fixed point loci -- away from
the stipulated points
$p_{[a_1,s_2,\ldots,a_n],j}$ -- provided $(M-1)(n+1) \ge  M$.
It only remains to consider the special case $M=1$. Here we take
$$W=(V_{[a_1,a_2,\ldots,a_n]} \oplus \bC)^2,$$
imposing conditions at just one point $(0:0:\cdots:0:1)$. 
Here $X\subset \bP(W)$ has codimension
$n+1$ and the fixed point loci are $\bP^1$'s, so we may avoid extraneous
points of intersection.
\end{proof}

\section{Refined invariants}
\label{sect:refined}

\subsection{Encoding fixed points}
Since $\bCB_2(C_2)=0$ this invariant says {\bf nothing}
about involutions of surfaces! Bertini, Geiser, and De Jonqui\`eres
involutions are perhaps the most intricate parts of the classification, which relies on 
the study of fixed curves. This leads to a natural refinement of the invariants:
For the symbols of type $[a,0]$, 
corresponding to curves $F_{\alpha}\subset X$ fixed by $C_N$,
we keep track of the (stable) birational equivalence class of $F_{\alpha}$
and the element of $\bCB_1(C_N)$ associated with $[a]$.

In general, \cite{kontsevichpestuntschinkel} introduced a group combining 
the purely number-theoretic information encoded in $\bCB_n(G)$ 
with geometric information encoded in the Burnside group of fields from \cite{KT}. 
Let
$$
\mathrm{Bir}_{n-1,m}(k), \quad  0\le  m\le  n - 1,
$$
be the set of  $k$-birational equivalence classes of $(n-1)$-dimensional irreducible varieties over $k$, 
which are $k$-birational to products $W\times \mathbb A^m$, 
but not to $W'\times \mathbb A^{m+1}$, for any $W'$, and put
\begin{equation}
\label{eqn:bnk}
\bCB_n(G,k):= \oplus_{m=0}^{n-1} \oplus_{[Y]\in \mathrm{Bir}_{n-1,m}(k)}\bCB_{m+1}(G).
\end{equation}
Let $X$ be a smooth projective variety of dimension $n$ 
with a regular, generically free, action of an abelian group $G$. 
Put
$$
\beta_k(X\actsfromright G):=\sum_{\alpha} \beta_{k,\alpha},
$$
where, as before, the sum is over components $F_{\alpha}\subset X^G$ 
of the $G$-fixed point locus, but in addition to the 
eigenvalues $a_1,\ldots, a_{n-\dim(F_{\alpha})}\in A$ in the tangent space $\mathcal T_{x_{\alpha}} X$ one keeps information about the function field of the component $F_{\alpha}$.
Choosing $m_{\alpha}$ so that
$$
[F_{\alpha}\times \mathbb A^{n-1-\dim(F_{\alpha})} ] \in \mathrm{Bir}_{n-1,m_{\alpha}}(k)
$$
we set
$$
\beta_{k,\alpha}:= [a_1,\ldots, a_{n-\dim(F_{\alpha})},  
\underbrace{0, \ldots, 0}_{m_{\alpha}+1-n+\dim(F_{\alpha})}] \in
\mathcal B_{m_{\alpha}+1}(G),
$$
identified with the summand of (\ref{eqn:bnk}) indexed by $[F_{\alpha}\times \mathbb A^{n-1-\dim(F_{\alpha})}]$.
When $F_{\alpha}$ is not uniruled we get a symbol
in $\bCB_{\bCodim(F_{\alpha})}(G)$.

\begin{theo} \cite[Remark 5]{kontsevichpestuntschinkel}
The class 
$$
\beta_k(X\actsfromright G)\in \bCB_n(G,k)
$$
is a $G$-birational invariant. 
\end{theo}


\subsection{Encoding points with nontrivial stabilizer}
\label{subsect:burndef}

We continue to assume that $G$ is a finite abelian group, 
acting regularly and generically freely on a smooth variety $X$. 
Let $H\subset G$ arise as the stabilizer of some point of $X$,
$F\subset X^H$ an irreducible component of the fixed point locus with generic stabilizer $H$, and $Y$ the minimal $G$-invariant subvariety containing $G$.
In Section~\ref{subsect:resolve}, we will explain how to blow up $X$ so that
$Y$ is always a disjoint union of translates of $F$.

Additional information about the action of $G$ on $X$ is contained in the action of the quotient 
$G/H$, which could act 
on the function field of $F$, or by translating $F$ in $X$. 
The paper \cite{kreschtschinkel} introduced the group
$$
\Burn_n(G)
$$
as the quotient by certain relations of the free abelian group generated by symbols
\begin{equation}
\label{eqn:symbol}
(H, G/H\actsfromleft K, \beta),
\end{equation}
where $K$ is a $G/H$-Galois algebra over a field 
of transcendence degree $d\le n$ over $k$, up to isomorphism, and $\beta$ is a faithful $(n-d)$-dimensional 
representation of $H$ (see \cite[Def.~4.2]{kreschtschinkel} for a precise formulation of conditions on $K$ and relations).   

Passing to a suitable $G$-equivariant smooth projective model $X$ 
-- as discussed in Section~\ref{subsect:resolve} --  its class is defined by
\begin{equation}
\label{eqn:xg}
[X\actsfromright G] :=\sum_{H\subseteq G} \sum_Y (H, G/H\actsfromleft k(Y),\beta_Y(X))\in \mathrm{Burn}_n(G),
\end{equation}
where the sum is over all $G$-orbits of components $Y\subset X$ with generic stabilizer $H$ as above, the symbol records
the eigenvalues of $H$ in the tangent bundle to $x\in Y$ as well as the $G/H$-action 
on the total ring of fractions $k(Y)$.

\begin{exam} \label{exam:Klein2}
We revisit Example~\ref{exam:Klein1} using the dual basis of characters
$\chi_1, \chi_2$ of $G=C_2\times C_2$. Here we have
\begin{align*}
[\bP^1 \actsfromright G] =& (\left<1\right>,G\actsfromleft k(t)) + 
(\left<g_1\right>,  G/\left<g_1\right> \actsfromleft \{ (1:\pm i)\},\chi_1) \\
 &+ (\left<g_2\right>, G/\left<g_2\right>  \actsfromleft \{ (1:0),(0:1)\},\chi_2) \\
 &+ (\left<g_1g_2\right>, G/\left<g_1g_2\right> \actsfromleft \{ (1:\pm 1)\},\chi_1+\chi_2).
\end{align*}
The action on $k(\bP^1)=k(t)$ 
is by $g_1(t)=-t$ and $g_2(t)=-1/t$. 
\end{exam}

Blowup relations ensure that $[X\actsfromright G]$ is a well-defined
$G$-birational invariant -- see Section~\ref{subsect:blowupsample}
for more details.

There is a distinguished subgroup
$$
\Burn_n^{\rm triv}(G) \subset \Burn_n(G)
$$
generated by symbols $(1, G\actsfromleft K=k(X))$. 
For `bootstrapping' purposes -- where we seek invariants of $n$-folds
in terms of lower-dimensional strata with nontrivial stabilizers --
we may suppress these tautological symbols to get a quotient 
$$
\Burn_n(G) \to\Burn_n^{\rm nontriv}(G).
$$
And there is also a 
natural quotient group
$$
\Burn_n(G) \to \Burn_n^G(G)
$$
obtained by suppressing all symbols $(H, G/H\actsfromleft K, \beta)$ where $H$ is a proper subgroup of $G$. 
By \cite[Prop. 8.1 and Prop. 8.2]{kreschtschinkel}, there are natural surjective homomorphisms 
$$
 \Burn_n^G(G) \to \bCB_n(G,k)\to \bCB_n(G).
$$
\begin{exam}
The group $\Burn_1^{\rm nontriv}(G)$ is freely generated by nontrivial subgroups $H\subset G$
and injective characters $a:H \hookrightarrow \bG_m$, e.g. $\Burn_1^{\rm nontriv}(C_N)\cong \bZ^{N-1}$
and $\Burn_1^{\rm nontriv}(C_2\times C_2)\cong \bZ^3$.
\end{exam}

\subsection{Examples of blowup relations}
\label{subsect:blowupsample}
We illustrate the relations for fixed points of cyclic actions on surfaces.
More computations are presented in Section~\ref{subsect:examples};
the reader may refer to
Section 4 of \cite{kreschtschinkel} for the general formalism.
All the key ideas are manifest in the surface case because the full set
of blowup relations follows from those in codimension two --
see \cite[Prop.~8.1]{kreschtschinkel} and the special case
Prop.~\ref{prop:81} above.

Suppose that $G=C_N$ acts on the surface $X$ with fixed point $\mathfrak p$ and
weights $a_1$ and $a_2$ that generate $A=\Hom(C_N,\bG_m)$. Let $\widetilde{X}$
denote the blowup of $X$ at $\mathfrak p$ and $E\simeq \bP^1$ the exceptional
divisor.
Let $\oH=\operatorname{ker}(a_1-a_2)\subset G$
denote the generic stabilizer of $E$ which acts faithfully on the
normal bundle ${\mathcal N}_{E/\widetilde{X}}$ via
$\bar{a}_1=\bar{a}_2 \in \oA:=\Hom(\oH,\bG_m)$. 

Assume first that $a_1$ and $a_2$ are both prime to $N$, so $\mathfrak p$ is
an isolated fixed point of $X$.
The quotient $G/\oH$
(when nontrivial) acts faithfully on $E=\bP^1$ with
 fixed points $\mathfrak p_1$ and $\mathfrak p_2$.
If $\oH=G$ then $a_1=a_2$ and
we get the relation
$$(G,{\rm triv} \actsfromleft k=k(\mathfrak p), (a_1,a_1))
=(G,{\rm triv} \actsfromleft k(t)=k(E), (a_1)).$$
If ${\rm triv } \subsetneq H'\subsetneq G$ then   
\begin{align*} 
(G,{\rm triv}& \actsfromleft k=k(\mathfrak p), (a_1,a_2))
=(\oH, G/\oH \actsfromleft k(t)=k(E), \bar{a}_1=\bar{a}_2 )\\
&+(G, {\rm triv} \actsfromleft k=k(\mathfrak p_2), (a_1,a_2-a_1))\\
& +(G, {\rm triv} \actsfromleft k=k(\mathfrak p_1), (a_2,a_1-a_2))
\stepcounter{equation}\tag{\theequation}\label{eqn:blow1}
\end{align*}
where $G/\oH$ acts on $t$ by a primitive $d$th root of unity
with $d=|G/\oH|$.
If $\oH$ is trivial then
\begin{align*}
(G,{\rm triv} \actsfromleft k=k(\mathfrak p), & (a_1,a_2))
=(G, {\rm triv} \actsfromleft k=k(\mathfrak p_2), (a_1,a_2-a_1)) \\
&+(G, {\rm triv} \actsfromleft k=k(\mathfrak p_1), (a_2,a_1-a_2)).
\stepcounter{equation}\tag{\theequation}\label{eqn:blow2}
\end{align*}

Assume now that $a_1=m_1n_1$ and $a_2=m_2n_2$ where $n_1,n_2|N$
(and are relatively prime modulo $N$)
and $m_1$ and $m_2$ are prime to $N$ and each other.
Then we have
$$\mathfrak{p} \in F_1 \cap F_2$$
for irreducible curves $F_1$ and $F_2$ with generic stabilizers
$C_{n_1}$ and $C_{n_2}$ respectively.
Let $\widetilde{F_1}, \widetilde{F_2} \subset \widetilde{X}$ denote
the proper transforms of $F_1$ and $F_2$ and $\mathfrak p_1,\mathfrak p_2 \in E$
their intersections with the exceptional divisors.
Thus the contribution to the strata containing $\mathfrak p$ is
\begin{align*}
(G, {\rm triv} \actsfromleft k(\mathfrak p), & (a_1,a_2)) \\
+(C_{n_1}, C_N/C_{n_1} \actsfromleft k(F_1), a_2)  
&+(C_{n_2}, C_N/C_{n_2} \actsfromleft k(F_2), a_1)
\end{align*}
with the latter two terms appearing in the symbol sum on
$\widetilde{X}$, with the $F_i$ replaced by the $\widetilde{F_i}$.
Note that $\oH=\operatorname{ker}(a_1-a_2)\subsetneq G$
because $a_1\not \equiv a_2$.
Here the blowup formula takes the form (\ref{eqn:blow1}) or (\ref{eqn:blow2})
depending on whether $H'$ is trivial or not.

Now suppose that $a_2=0$.
Let $F\subset X$ denote the irreducible component of the fixed locus
containing $\mathfrak p$, so that $a_1$ is the character by which
$G$ acts on ${\mathcal{N}}_{F/X}$. Write $\widetilde{F} \subset \widetilde{X}$
for the proper transform of $F$, $\mathfrak p_1=\widetilde{F} \cap E$,
and $\mathfrak p_2\in E$ for the other fixed point.
Here we get the relation
$$(G,G\actsfromleft k(F),a_1)=
(G,G\actsfromleft k(F),a_1) + (G,G\actsfromleft k(\mathfrak p_2),(a_1,-a_1)),$$
whence the latter term vanishes.

\subsection{Examples}
\label{subsect:examples}
We now complement the computations in Section~\ref{sect:surf}, for $G=C_N$, and small $N$. As before, we work over an algebraically-closed based field $k$
of characteristic zero. 

\

\noindent
($N=2$)
\begin{itemize}
\item $\bCB_2(C_2)=0$.
\item $\bCB_2(C_2,k) \cong \Burn_2^{C_2}(C_2)$; has a copy of
$\bCB_1(C_2)=\bZ$, for every isomorphism class of curves of positive genus.
\item $\Burn_2(C_2)=\Burn_2^{\mathrm{triv}}(C_2)\oplus \Burn_2^{C_2}(C_2)$.
\end{itemize}

\

\noindent
($N=3$)
\begin{itemize}
\item $\bCB_2(C_3)\cong \bZ$, generated by $[1,1]$,
\[ [1,2]=0, \qquad [2,2]=-[1,1]. \]
\item 
$\bCB_2(C_3,k)\cong \Burn_2^{C_3}(C_3)$, is a direct sum of $\bZ$, corresponding to
points and rational curves, and a copy of
$\bCB_1(C_3)\cong \bZ^2$ for every isomorphism class of curves of positive genus.
\item 
$\Burn_2(C_3)=\Burn_2^{\mathrm{triv}}(C_3)\oplus \Burn_2^{C_3}(C_3)$.
\end{itemize}

\

\noindent
($N=4$)
\begin{itemize}
\item $\bCB_2(C_4)\cong \bZ$, generated by $[1,2]$,
\[ [1,1]=2[1,2],\quad [1,3]=0,\quad [2,3]=-[1,2],\quad [3,3]=-2[1,2]. \]
\item $\bCB_2(C_4,k)$ is a direct sum of $\bZ$, corresponding to points
and rational curves, and a copy of $\bCB_1(C_4)\cong\bZ^2$ for every
isomorphism class of curves of positive genus.
\item $\Burn_2(C_4)=\Burn_2^{\mathrm{triv}}(C_4)\oplus \Burn_2^{\mathrm{nontriv}}(C_4)$:
$\Burn_2^{\mathrm{nontriv}}(C_4)$ has,
for every isomorphism class of curves of positive genus,
a copy of $\bCB_1(C_4)$ and a copy of $\bCB_1(C_2)$, with an
additional copy of $\bCB_1(C_2)$ for every
isomorphism class of curves of positive genus with involution.

We claim points and rational curves
contribute 
$$\bZ^2 \subset \Burn_2^{\mathrm{nontriv}}(C_4),$$ 
generated by $[1,2]$ and $[2,3]$ where
$$[i,j]=(C_4,k,(i,j)), \quad (i,j)=(1,1),(1,2),(1,3),(2,3),(3,3).$$
Abusing notation, write
$$[1,0] = (C_4,k(t),1) \quad
        [3,0] = (C_4,k(t),1).$$
We write down the blowup relations, both orbits of points with special
stabilizers and orbits on one-dimensional strata 
with nontrivial stabilizer:
\begin{align*}
        [0,1]&=[0,1]+[1,3] \\
        [0,3]&=[0,3]+[1,3] \\
        [1,1]&=[1,0] \\
        [1,2]&=[1,1]+[2,3] \\
        [1,3]&=[1,2]+[2,3]+(C_2,C_2 \actsfromleft k(t), 1) \\
        [2,3]&=[1,2]+[3,3] \\
        [3,3]&=[3,0] \\
     (C_2,C_2 \actsfromleft k^2, (1,1))  &=(C_2,C_2 \actsfromleft k(t)^2, 1) \\
   (C_2,C_2 \actsfromleft k(t), 1)&=(C_2,C_2 \actsfromleft k(t), 1)+(C_2,C_2 \actsfromleft k(t)^2, 1) \\
            (C_2,C_2 \actsfromleft k(t)^2, 1)&=(C_2,C_2 \actsfromleft k(t)^2, 1)+(C_2,C_2 \actsfromleft k(t)^2, 1)
\end{align*}
Thus we find
\begin{gather*}
[1,3]=0, \\
[0,3]=[3,3]=-[1,2]+[2,3],\\
[0,1]=[1,1]=[1,2]-[2,3], \\
(C_2,C_2\actsfromleft k(t),1)=-[1,2]-[2,3], \\
(C_2,C_2\actsfromleft k^2(t),1)=0, \\
(C_2,C_2\actsfromleft k^2,(1,1))=0.
\end{gather*}
\end{itemize}
Here $k^n$ denotes the total ring of fractions for an orbit of length
$n$ and $k^n(t)$ the total ring of fractions of the exceptional locus
of the blowup of such an orbit.
Furthermore, $C_2\actsfromleft k(t)$ is via $t\mapsto -t$.

For example, the linear action on $\bP^2$
$$(x:y:z) \mapsto (x:iy:-iz)$$
has invariant
$$[1,3]+[1,2]+[2,3]+(C_2,C_2 \actsfromleft k(y/z), 1)=0
\in \Burn_2^{\mathrm{nontriv}}(C_4).$$

\

We close this section with a noncyclic example: 

\

\noindent
($C_2\times C_2$) \newline
Write
$$G=C_2\times C_2=\{ 1,g_1,g_2,g_3=g_1g_2\}$$
and 
$$G^{\vee} =  \{0,\chi_1,\chi_2,\chi_3=\chi_1+\chi_2\}, \,
\chi_1(g_1)=\chi_2(g_2)=-1,\chi_1(g_2)=\chi_2(g_1)=1$$
as before.
\begin{itemize}
\item $\bCB_2(C_2\times C_2)=
\{0,[\chi_1,\chi_2],[\chi_1,\chi_3],[\chi_2,\chi_3]\}$
with the structure of the Klein four group presented in
Section~\ref{sect:surf}.
\item $\bCB_2(C_2\times C_2,k) \cong \bCB_2(C_2\times C_2)$ 
as $\bCB_1(C_2\times C_2)=0$.
\item $\Burn_2(C_2\times C_2)=\Burn^{\rm triv}(C_2\times C_2)
\oplus \Burn_2^{\rm nontriv}(C_2 \times C_2)$
where the second term is a direct sum of 
the subgroup $R$ generated by points and rational curves,
a copy of 
$$\Burn_1^{\rm nontriv}(C_2\times C_2)\cong \bZ^3$$
for each curve of positive genus, and another copy
for each curve of positive genus equipped with an involution. 
\end{itemize}

The group $R$ fits into an exact sequence
$$0 \ra \Burn_1^{\rm nontriv}(C_2\times C_2) \ra R \ra \bCB_2(C_2\times C_2) \ra 0$$
obtained by computing generators and relations.

\

\paragraph{{\bf Generators:}}
Here we take $1\le i < j \le 3$:
\begin{align*}
[\chi_i,\chi_j] &:=(G, {\rm triv} \actsfromleft k, (\chi_i,\chi_j)) \\
e_i &:=(\left<g_i\right>, G/\left<g_i\right> \actsfromleft k(t), \bar{\chi}_i=1)\\
q_i &:=(\left<g_i\right>, G/\left<g_i\right> \actsfromleft k^2, (\bar{\chi}_i,\bar{\chi}_i)=(1,1)) \\
f_i &:=(\left<g_i\right>, G/\left<g_i\right> \actsfromleft k^2(t), \bar{\chi}_i=1)
\end{align*}

\paragraph{{\bf Relations:}}
Here we choose $h$ so that $\{h,i,j\}=\{1,2,3\}$:
\begin{align*}
[\chi_i,\chi_j] &= e_h + [\chi_h,\chi_i] + [\chi_h,\chi_j] \quad
				  \text{(blow up fixed point)} \\
q_i &= f_i \quad  \text{(blow up orbit $q_i$)} \\
e_i & = e_i + f_i \quad  \text{(blow up general orbit of $e_i$)} 
\end{align*}
Thus the $q_i$ and $f_i$ are zero and we have
$$R/\left<e_1,e_2,e_3\right> = \bCB_2(C_2\times C_2).$$

We revisit the actions of $C_2\times C_2$
on rational surfaces in Section~\ref{sect:surf} using these new techniques:
\begin{enumerate}
\item[(1)]{the action $(x,y) \mapsto (\pm x^{\pm 1},y)$ on $\bP^1\times \bP^1$
has invariant 
$$
f_1+f_2+f_3=0;
$$
}
\item[(2)]{the action $(x,y) \mapsto (\pm x,\pm y)$ on $\bP^1 \times \bP^1$
has invariant 
$$
2e_1+2e_2+4[\chi_1,\chi_2]=0;
$$}
\item[(3)]{the diagonal action on $\bP^1\times \bP^1$ has invariant
$$
2(q_1+q_2+q_3)=0;
$$}
\item[(4)]{the action on the conic fibration admits an elliptic curve
$$F=\{y_1^2(x_1-ax_2)(x_1+ax_2)=y_2^2(x_1-bx_2)(x_1+bx_2)\}$$
that is fixed by $g_2$ and fibers over $(x_1:x_2)=(0:1),(1:0)$
fixed by $g_1$ and thus has invariant
$$4[\chi_1,\chi_2]+2e_1 + 
(\left<g_2\right>,\left<g_1\right>\actsfromright k(F),1) \neq 0$$
where $g_1$ acts on $k(F)$ by $x_1/x_2 \mapsto -x_1/x_2$;}
\item[(5)]{the action on the degree two del Pezzo surface has
nontrivial invariant arising from the positive genus curves
fixed by $g_1$ and $g_2$;}
\item[(6)]{the action on the degree one del Pezzo surface has 
nontrivial invariant arising from the positive genus curves
fixed by the involution.}
\end{enumerate}

\subsection{Limitation of the birational invariant}

It is an elementary observation, recorded in \cite[App.\ A]{reichsteinyoussinessential},
that the presence of a
point fixed by a given abelian subgroup $H$ of $G$ is a birational invariant of
a smooth projective variety $X$ with generically free $G$-action.
Two smooth $n$-dimensional projective varieties with generically free $G$-action might be
distinguished in this way but nevertheless have the same class in
$
\Burn_n^{\rm nontriv}(G).
$

Indeed, letting $C_2$ act on $\bP^1$, we consider the corresponding
product action of $C_2\times C_2$ on $\bP^1\times \bP^1$.
As well, the action of $C_2\times C_2$ on
$\bP^1$ gives rise to a diagonal action on $\bP^1\times \bP^1$.
The former, but not the latter, has a point fixed by $C_2\times C_2$,
so the actions belong to two distinct birational classes.
However, both actions give rise to a vanishing class in
$\Burn_2^{\rm nontriv}(C_2\times C_2)$.

\subsection{Reprise: Cyclic groups on rational surfaces}

As already discussed in Section~\ref{sect:surf}, 
the presence of higher genus curves in the fixed locus of the action of a 
cyclic group of prime order on a rational surface is an important invariant 
in the study of the plane Cremona group. For example, for $G=C_2$, these curves
make up entirely the group $\Burn_2^{G}(G)$ 
and famously characterize birational involutions of the plane up to conjugation 
\cite{baylebeauville}.

For more general cyclic groups acting on rational
surfaces, we recover the NFCA invariant of Blanc \cite{blancsubgroups},
which governs his classification. 

We recall the relevant definitions: 
Let $g\in \mathrm{Cr}_2$ be a nontrivial element of the plane Cremona group, of finite order $m$. 
\begin{itemize}
\item {\em Normalized fixed curve} \cite{deFer}:
$$
\mathrm{NFC}(g):=\begin{cases} 
 \text{isomorphism class of the normalization  of the union}  \\
 \text{of irrational components of the fixed curve.}
 \end{cases}
$$
\item {\em Normalized fixed curve with action}:
$$
\mathrm{NFCA}(g):=\left((\mathrm{NFC}(g^r ), g\mid_{\mathrm{NFC}(g^r)}  \right)^{m-1}_{r=1},
$$
where $g\mid_{\mathrm{NFC}(g^r)} $ is the automorphism induced by $g$ on $\mathrm{NFC}(g^r )$.
\end{itemize}
One of the main results in this context is the following characterization:

\begin{theo}\cite{blancsubgroups} 
Two cyclic subgroups $G$ and $H$ of order $m$ of  $\mathrm{Cr}_2$ are conjugate if and only if 
$$
\mathrm{NFCA}(g) = \mathrm{NFCA}(h),
$$
for some generators $g$ of $G$ and $h$ of $H$. 
\end{theo}

It follows immediately from the definition that the information encoded in 
$\Burn_2(G)$, for $G=C_m$, is {\em equivalent} to $\mathrm{NFCA}(g)$.  

\begin{rema}
It would be interesting to use symbol invariants to organize the
classification of 
cyclic group actions on rational threefolds.  
\end{rema}

\section{Cubic fourfolds}
\label{sect:cubic4}

In this section we discuss several illustrative examples, showing 
various aspects of the new invariants introduced above. 
Equivariant geometry in dimension $\le 3$ is, in principle, 
accessible via the Minimal Model Program, 
and there is an extensive literature on factorizations of equivariant birational maps. 
We focus on dimension four, and in particular, 
on cubic fourfolds.

Let $X\subset \bP^5$ be a smooth cubic fourfold. No examples are
known to be irrational! 
Here we show that there are actions $X\actsfromright G$ 
where $G$-equivariant rationality fails, including actions on rational cubic fourfolds. 

We found it useful to consult the classification of possible abelian automorphism 
groups of cubic fourfolds in \cite{may}.   
Here is a list of $N>1$, such that the cyclic group $C_N$ acts on a smooth cubic fourfold:
$$
N= 2,3,4,5,6,8,9,10,11,12,15,16,   18, 21,  24,30,32,33,36, 48.
$$
Note that 
$$
\mathcal B_4(C_N)\otimes \mathbb Q =0, \quad  \text{ for all $N < 27$, $N=30,32$}.
$$ 
We record their dimensions $d_{\bQ}$ in the remaining cases:
$$
\begin{array}{c|ccc}
N           & 33 & 36 & 48 \\
\hline 
d_{\bQ} &  2  &  3  & 7 
\end{array}
$$
One can also work with finite coefficients: let 
$$
d_p=d_p(N):=\dim \mathcal B_4(C_N)\otimes \mathbb F_p,
$$ 
we have $d_2,d_3,d_5=0$, for all $N\le 15$, and $N=18,21$. In the other cases, we find: 

$$
\begin{array}{c|ccccccc}
N    & 16 & 24 & 30 & 32 & 33 & 36 & 48\\
\hline 
d_2 &  1  &  5  & 10  &  12 & 3 &  19 & 50 \\
d_3 &   0    &  0    &   0    &  0     & 2 & 3 & 7  \\
d_5  &    0   &   0   &     0    &   0    &  2& 3& 7  \\
d_7  &      0 &    0  &     0    &    0   &  2  & 3 & 7
\end{array}
$$

\

Thus, to exhibit applications of $\bCB_4(C_N)$ we have to look at large $N$.

\

\noindent
{\bf Using $\mathcal B_n(G)$:} Consider $X\subset \bP^5$ given by
\begin{equation}
\label{eqn:36}
 x_1^2x_2+x_2^2x_3+x_3^2x_1+x_4^2x_5+x_5^2x_0+x_0^3 =0.
\end{equation}
It carries the action of $G=C_{36}$, with weights  
$$
(0,4,28,16,9,18)
$$ 
on the ambient projective space, 
and isolated $G$-fixed points
$$
P_1=(0:1:0:0:0:0), \ldots , P_5:=(0:0:0:0:0:1).
$$ 
Computing the weights in the corresponding tangent spaces, we find that 
$\beta(X)$ equals
$$
[4,24,31,22]+[28,24,19,10]+[24,12,7,34]+[9,5,17,29]+[14,26,2,9].
$$
Solving a system of 443557 linear equations in 82251 variables, we find, by 
computer, that 
$$
\beta(X)\neq \beta(\bP^4\actsfromright C_{36} )=    0 \in \mathcal B_4(C_{36})\otimes \mathbb F_2=\mathbb F_2^{19}.
$$
This implies that $X$ is not $G$-equivariantly birational to $\bP^4$. 

\

\

\noindent
{\bf Using co-multiplication:} 
The fourfold $X\subset \bP^5$ given by
\begin{equation}
\label{eqn:48}
x_2^2x_3+x_3^2x_4+x_4^2x_5+ x_5^2x_0 + x_0^3 + x_1^3=0
\end{equation}
carries an action of $C_{48}$ with weights $(0,16,3,-6,12,-24)$, and isolated fixed points. 
We find that $\beta(X)\in \bCB_{4}(C_{48})$ equals
$$
[-3,13,9,-27]+[6,22,9,-18]+[-12,4,-9,-18]+[40,27,18,36].
$$
Here, we apply the co-multiplication formula from Section~\ref{subsect:mult}
to the image $\beta^-(X)$ of $\beta(X)$ under the projection 
$$
\bCB_{4}(C_{48}) \ra \bCB_{4}^-(C_{48}).
$$
Let 
$$
G':=\bZ/3\bZ, \quad  G'':=\bZ/16\bZ, \quad n'=1\quad \text{ and } \quad n''=3.
$$ 
We have a homomorphism 
$$
\nabla^-: \bCB_{4}^-(C_{48}) \to \bCB_{1}^-(C_{3}) \otimes  \bCB_{3}^-(C_{16}) 
$$
and we find  that $\nabla^-(\beta^-(X))$ equals  
$$
 [1]^-\otimes ([-1,3,-9]^- + [2,3,-6]^- + [-4,-3,-6]^- + [9,6,12]^-).
$$ 
Now we are computing in $\bCB_{3}^-(C_{16})$, a much smaller group. 
We have 
$$
\dim \bCB_{3}(C_{16})\otimes \bQ  =3, \text{ but } \dim \bCB_{3}^-(C_{16})\otimes \bQ  =0,
$$
however
$$
\dim \bCB_{3}(C_{16})\otimes \bF_2  =8 \text{ and } \dim \bCB_{3}^-(C_{16})\otimes \bF_2  =7.
$$
We find, by computer, that 
$$
[-4,-3,-6] = [9,6,12] \in \bCB_{3}(C_{16})\otimes \bF_2,  
$$
so that  the sum of these terms does not contribute to $\nabla^-(\beta^-(X))$, 
and check that 
$$
[-1,3,-9]^- + [2,3,-6]^- = [1,2,10]^-\neq 0 \in \bCB_{3}^-(C_{16}).
$$
It follows that 
$$
\nabla^-(\beta^-(X)) \neq 0, \quad \text{ thus} \quad \beta^-(X)\neq 0 \in \bCB_{4}^-(C_{48}),
$$
and this action of $C_{48}$ on the cubic fourfold is not equivariantly birational to a linear action on $\bP^4$, by Proposition~\ref{prop:cn}.

\

\noindent
{\bf Using $\mathcal B_{n}(G,k)$:}
There is also another way to analyze the fourfold in \eqref{eqn:48}:
observe that the divisor $Y\subset X$, a smooth cubic threefold given by 
$x_1=0$, is fixed by $C_3 \subset C_{48}$. 
This divisor is irrational, and we get a nontrivial contribution to 
$\beta_k(X)\in \bCB_{4}(C_3,k)$, in the summand labeled by 
$Y\in \mathrm{Bir}_{1,0}$; thus $X$ is not even $C_3$-equivariantly birational to $\bP^4$.

\

The fourfold $X\subset \bP^5$ given by
$$
 f_3(x_0,x_1,x_2)+ x_3^2x_4+x_4^2x_5+x_5^2f_1(x_0,x_1,x_2)=0
$$
carries the action of $G=C_8$, with weights
$$
(0,0,0,1, 6, 4).
$$
The fourfold is smooth, e.g., for $f_3=x_0^3+x_1^3+x_2^3$ and $f_1=x_0$. 

Here, there are no isolated fixed points, but we find 
information from fixed point loci in higher dimension.   
The $G$-fixed locus contains the degree 3 curve $Y$ given by
$$
 x_3=x_4=x_5=f_3(x_0,x_1,x_2)=0,
 $$
which is smooth for appropriate $f_3$. 
Thus we get a contribution to 
$\beta_k(X)\in \bCB_{4}(G,k)$, in the summand labeled by 
$Y\in \mathrm{Bir}_{3,2}$:
$$
[7,2,4] \neq 0 \in \bCB_3(C_8)\cong \mathbb F_2.
$$ 
Here, we solve 289 linear equations in 120 variables.  
This implies that $X$ is not $G$-equivariantly birational to $\bP^4$. Of course, 
this also follows by observing that the fourth power of the generator fixes a cubic threefold.

\

\noindent
{\bf Using $\mathrm{Burn}_{n}(G)$:}
Consider $X\subset \bP^5$ given by
\begin{equation} \label{C6cubic}
x_0x_1^2+x_0^2x_2-x_0x_2^2-4x_0x_4^2+x_1^2x_2+x_3^2x_5-x_2x_4^2-x_5^3=0.
\end{equation}
It carries the action of $G=C_6$, which acts with weights
\[ (0,0,0,1,3,4). \]
The cubic fourfold $X$ is rational,
since it contains the disjoint planes
$$
x_0=x_1-x_4=x_3-x_5=0\quad \text{ and } \quad 
x_2=x_1-2x_4=x_3+x_5=0.
$$
Noticing a cubic surface $S\subset X$, with $C_3$-stabilizer and
scalar action on the normal bundle, the fact that the $C_2$-action on $S$
fixes an elliptic curve on it lets us conclude, by \cite{BogPro}, 
that the cubic surface is not stably $C_2$-equivariantly rational; the corresponding symbol
$$
[ C_3, C_2\actsfromleft k(S),\beta] \neq 0 \in \Burn_4(C_6),
$$
moreover, it does not interact with any other symbols in 
$[X\actsfromright G]$, which implies that $X$ is not $G$-birational to $\bP^4$ with linear action.  
In this case, no subgroup of $C_6$ fixes a hyperplane section. 

We discuss obstructions of such type in Section~\ref{sec.elementary} below -- formally, 
$\Burn_4(C_6)$ admits a projection to $\bZ$ that distinguishes
the equivariant birational class of $X$ from that of
$\bP^4$ with linear action.

\

Kuznetsov \cite{Kuz} conjectures which cubic fourfolds $X$
are rational, in terms of their derived categories
$\mathsf D^b(X)$.
Consider the line bundles $\bCO_X,\bCO_X(1),\bCO_X(2)$ and the right
orthogonal complement
\begin{equation} \label{derived}
\bCA_X = \left< \bCO_X,\bCO_X(1),\bCO_X(2)\right>^{\perp}.
\end{equation}
Conjecturally, $X$ is rational if and only if $\bCA_X \cong \mathsf D^b(S)$,
the bounded derived category of a K3 surface $S$.
However, the K3 surface need not be canonically determined as there
are many examples of derived-equivalent but nonisomorphic K3 surfaces.
Indeed the following conditions on complex
projective K3 surfaces $S_1$ and $S_2$
are equivalent \cite{Orlov}
\begin{itemize}
\item{$\mathsf D^b(S_1)\cong \mathsf D^b(S_2)$;}
\item{the Mukai lattices
$$\widetilde{H}(S_1,\bZ) \cong \widetilde{H}(S_2,\bZ)$$
as Hodge structures;}
\item{the transcendental cohomology lattices 
$$H^2_{\tran}(S_1,\bZ) \cong H^2_{\tran}(S_2,\bZ)$$
as Hodge structures.}
\end{itemize}

There is an alternative Hodge-theoretic version of the conjecture:
A smooth cubic fourfold $X$ is rational if and only if there exists
a K3 surface $S$ and an isomorphism of integral Hodge structures
\begin{equation} \label{hodge}
H^4(X,\bZ)_{\tran}\cong H^2(S,\bZ)_{\tran}(-1),
\end{equation}
where $\tran$ denotes the orthogonal complement
of the Hodge classes and $(-1)$ designates the Tate twist.
Work of Addington and Thomas \cite{AdTh}, and recent extensions
\cite[Cor.~1.7]{BLMNPS}, show that the conditions (\ref{derived})
and (\ref{hodge}) are equivalent. In particular, both
are stable under specialization, consisting of an explicit countable
collection of divisors in moduli \cite{Has00}.
The main theorem of \cite{KT} -- that rationality is stable under
specializations of smooth projective varieties -- 
gives the equivalence of Kuznetsov's conjecture with 
the Hodge-theoretic statement. 

Suppose then that $X$ admits an action of a finite group $G$.
If $X$ is rational -- and the conjectures are true -- then
$G$ naturally acts on $\bCA_X$ and $\mathsf D^b(S)$,
for each surface $S$ arising in (\ref{derived}). 
There is an induced action on $\widetilde{H}^*(S,\bZ)$ as well.
It is natural to speculate that a $G$-equivariant birational
map $\bP^4 \stackrel{\sim}{\dashrightarrow} X$ should imply that 
we may choose $S$ in its derived equivalence class so that 
the $G$-action on the Mukai lattice is induced by a $G$-action on
$S$. 

There are several possible obstructions to finding such an $S$:
\begin{itemize}
\item{if $S$ exists then there exists a sublattice of algebraic
classes 
$$U := \left( \begin{matrix} 0 & 1 \\
			    1 & 0 \end{matrix} \right)=H^2(S,\bZ)^{\perp}$$
in the $G$-invariant part of the abstract
Mukai lattice arising from $\bCA_X$;}
\item{the action of $G$ on $\Pic(S)$ preserves the ample cone
of $S$.}
\end{itemize}
The first condition fails when $G$ permutes various derived equivalent 
K3 surfaces. The second condition fails if $G$ includes a Picard-Lefschetz
transformation associated with a smooth rational curve $\bP^1 \subset S$.
Derived equivalent K3 surfaces might have very different 
automorphism groups \cite[Ex.~23]{HT17}; this paper discusses
descent of derived equivalence in the presence of Galois actions.

We mention some results on when the group action can be
lifted to the associated K3 surface \cite[\S 8]{Ouchi}:
\begin{itemize}
\item{if $G\neq \{1\}, G$ acts on $X$ symplectically, i.e., acts trivially
on $\mathrm H^1(\Omega^3_X)$, then $S$ is unique;}
\item{if $X$ is the Klein cubic fourfold
$$
x_0^3 + x_1^2 x_2 + x^2_2 x_3 + x^2_3 x_4 + x^2_4 x_5 + x^2_5 x_1 = 0
$$
then $X$ admits a symplectic automorphism of order $11$ and $\bCA_X\cong \mathsf D^b(S)$
for a unique K3 surface, which has no automorphism of order $11$.}
\end{itemize}
We speculate that the Klein example should not be $C_{11}$-equivariantly
rational, even though
$$
\beta(X\actsfromright C_{11})=0 \in \bCB_4(C_{11}),
$$
as the $C_{11}$-action
has isolated fixed points and the target group is trivial
\cite[\S 8]{kontsevichpestuntschinkel}. 

\begin{ques}
Let $X$ be a smooth cubic fourfold with the action of a (finite) group $G$. 
Suppose that $\bCA_X \cong \mathsf D^b(S)$ for a K3 surface $S$ with $G$-action,
compatible with the isomorphism. Does it follow that
$$
[X\actsfromright G]=[\bP^4\actsfromright G] \in \Burn_4(G),
$$
for some action of $G$ on $\bP^4$? 
\end{ques}

It is mysterious how the invariants in the Burnside groups interact with
the actions on the Hodge structures on the middle cohomology of $X$. 
Obstructions to $G$-equivariant rationality arise from fixed loci in
various dimensions but the Hodge theory encodes codimension-two
cycles only. The example \eqref{C6cubic}, which is rational but not
$C_6$-rational, is particularly striking to us: How is
the cubic surface in the fixed locus coupled with the associated K3
surfaces?

\section{Nonabelian invariants}
\label{sect:nonab}

In this section, $G$ is a finite group, not necessarily abelian. 

\subsection{The equivariant Burnside group}
As in Section~\ref{subsect:burndef}, it is defined as the quotient of the $\bZ$-module generated by 
symbols
$$
(H, N_G(H)/H\actsfromleft K, \beta),
$$
similar to those in \eqref{eqn:symbol}, by {\bf blow-up relations}. The required relations are a bit
complicated but similar in spirit to what was written above; 
precise definitions are in \cite[Section 4]{kreschtschinkel}. 
The resulting group 
$$
\Burn_n(G)
$$
carries a rich combinatorial structure, that remains largely unexplored.

\subsection{Resolution of singularities} 
\label{subsect:resolve}
The class of $X\actsfromright G$, a projective variety with a generically free $G$-action, 
is computed on a suitable model 
of the function field $k(X)$. We explain how such a model may be found in practice.
While this is a corollary of Bergh's `destackification' procedure \cite{bergh},
the approach here can be helpful in specific examples.

We first review the resolution process of \cite[\S 3]{reichsteinyoussinessential}.
A variety with group action as above is {\em in standard form with respect
to a $G$-invariant divisor $Y$} if
\begin{itemize}
\item{$X$ is smooth and $Y$ is a simple normal crossings divisor;}
\item{the $G$ action on $X\setminus Y$ is free;}
\item{for every $g\in G$ and irreducible component $Z\subset Y$ either
$g(Z)=Z$ or $g(Z)\cap Z=\emptyset$.}
\end{itemize}
We recall several fundamental results. First, we can always
put actions in standard form:
\begin{quote}
If $X$ is smooth and $Y$ is a $G$-invariant closed subset
such that $G$ acts freely on $X\setminus Y$ then there exists a 
resolution
$$\pi:\widetilde{X} \ra X$$
obtained as a sequence of blowups along smooth $G$-invariant
centers, such that
$\widetilde{X}$ is in standard form with respect to 
$\operatorname{Exc}(\pi)\cup \pi^{-1}(Y)$ \cite[Th.~3.2]{reichsteinyoussinessential}.
\end{quote}

An action in standard form has stabilizers of special type:
\begin{quote}
Assume that $X$ is in standard form with respect to $Y$
and $x\in X$ lies on $m$ irreducible components of $Y$
then the stabilizer $H$ of $x$ is abelian with $\le m$ generators 
\cite[Th.~4.1]{reichsteinyoussinessential}.
\end{quote}
The proof of \cite[Th.~4.1]{reichsteinyoussinessential} (see Remark~4.4) 
yields \'etale-local coordinates on $X$ about $x$
$$x_1,\ldots,x_k,y_1,\ldots,y_l,z_1,\ldots,z_m$$
such that
\begin{itemize}
\item{$H$ acts diagonally on all the coordinates;}
\item{$y_1=\cdots=y_l=z_1=\cdots=z_m=0$ coincides with $X^H$, i.e., these are the coordinates on which $H$ acts nontrivially;}
\item{$y_1\cdots y_m=0$ coincides with $Y$ and the associated
characters of $\chi_i:H\ra \bG_m$ generate $\Hom(H,\bG_m)$ so
the induced representation
\begin{equation}
(\chi_1,\ldots,\chi_m):H \hookrightarrow  \bG_m^m \label{inject}
\end{equation}
is injective.}
\end{itemize}

Suppose that $X$ is in standard form with respect to $Y$
with irreducible components $Y_1,\ldots,Y_s$. For each orbit
of these under the action of $G$, consider the reduced divisor
obtained by summing over the orbit. The resulting divisors
$D(1),\ldots,D(r)$ have smooth support -- by the definition of 
standard form -- and
$$D(1) \cup \cdots \cup D(r) = Y_1 \cup \cdots \cup Y_s.$$
The line bundles $\bCO_X(D(i))$ are naturally $G$-linearized and thus
descend to line bundles on the quotient stack $[X/G]$ and we obtain
$$\varphi: [X/G] \rightarrow B\bG_m \times \cdots \times B\bG_m \quad (r \text{ factors}).$$

We claim $\varphi$ is representable. It suffices to check this by showing
that the induced homomorphism of stabilizers is injective at each point
\cite[\href{https://stacks.math.columbia.edu/tag/04YY}{Tag 04YY}]{stacks-project}.
For $x\in X$, fix the indices $i_1,\ldots,i_m$ so that 
$D(i_1),\ldots,D(i_m)$ are the components of $Y$ containing $x$,
and consider the induced
$$\varphi_x: [X/G] \rightarrow B\bG_m \times \cdots \times B\bG_m \quad (m \text{ factors}).$$
The homomomorphism on stabilizers is given by (\ref{inject}) which
is injective.

Thus we have established:
\begin{prop}
Let $X$ be a smooth projective variety with a generically free action by a 
finite group $G$, in standard form with respect to a divisor $Y$.
Then Assumption 2 of \cite{kreschtschinkel} holds and the invariants
constructed there may be evaluated on $X$.
\end{prop}

\subsection{The class of $X\actsfromright G$}
On a suitable model $X$, 
we consider each stratum $F\subset X$ with nontrivial (abelian)
stabilizer $H\subset G$, 
the action of the normalizer $N_G(H)/H$ on (the orbit of) the stratum, and the
induced action of $H$ on its normal bundle, and record these in a symbol. Then we 
define
\begin{equation}
\label{eqn:defn-xg}
[X\actsfromright G]:=\sum_{H\subseteq G} \sum_F(H,N_G(H)/H\actsfromleft k(F), \beta) \in \Burn_n(G), 
\end{equation}
as \eqref{eqn:xg}. The proof that this is a $G$-birational invariant relies on $G$-equivariant Weak Factorization and combinatorics \cite[Section 5]{kreschtschinkel}.

\subsection{Elementary observations}
\label{sec.elementary}
As we discussed, the presence of higher genus curves in the fixed locus of the
action of a cyclic group of prime order on a rational surface is an
important invariant in the study of the plane Cremona group;
see, e.g., \cite{blancsubgroups}.
These make up entirely the group $\Burn_2^{\bZ/2\bZ}(\bZ/2\bZ)$ and entirely
characterize birational involutions of the plane
up to conjugation \cite{baylebeauville}.

More generally, for any nontrivial cyclic subgroup $H$ of $G$ and birational class of an
$(n-1)$-dimensional variety $Y$, not birational to
$Z\times \bP^1$ for any variety $Z$ of dimension $n-2$,
we have a projection from $\Burn_n(G)$ onto
the free abelian group on the $N_G(H)$-conjugacy classes of pairs
$(H',a)$, where $H'$ is a subgroup,
$H\subset H'\subset N_G(H)$, and $a\in H^\vee$ is a primitive character.
This sends 
$$
(H,\Ind_{H'/H}^{N_G(H)/H}(k(Y)),a),
$$ 
for any
$H'/H\actsfromleft k(Y)$, to the
generator indexed by the conjugacy class $[(H',a)]$ of the pair.

A more refined version of this observation might also relax the
restriction on $Y$ but take into account the action $H'/H\actsfromleft k(Y)$.
We do not go into details, but only point out, for instance, that for $n=2$
and $Y=\bP^1$
there is a projection
\[ \Burn_2(G)\to \bigoplus_{\substack{[(H',a)]\\ H'/H \text{ not cyclic}}} \bZ. \]
Taking $G$ to be the dihedral group of order $12$ and $H$ the center of $G$,
we may distinguish between the two inclusions of $G$ into the
plane Cremona group considered in \cite{isk-s3}, see Section~\ref{subsect:isk} below.

\subsection{Dihedral group of order 12.} 
\label{subsect:relations}
We now compute $\Burn_2(G)$, for $G=D_{6}$, 
the dihedral group with generators 
$$\rho, \quad \sigma,\quad \text{  with } \quad 
\rho^6=\sigma^2=\rho\sigma\rho\sigma=e_{D_6}.
$$
We list abelian subgroups, up to conjugacy:
\begin{itemize}
\item order $6$: $C_6=\langle \rho\rangle$
\item order $4$: $D_2=\langle \rho^3,\sigma\rangle$
\item order $3$: $C_3=\langle \rho^2\rangle$
\item order $2$: central $C_2=\langle \rho^3\rangle$, noncentral
$S:=\langle \sigma\rangle$, $S':=\langle \rho^3\sigma\rangle$
\item order $1$: $\mathrm{triv}$
\end{itemize}
The subgroup of order $4$ and two noncentral subgroups of order $2$ have
normalizer $D_2$, the others are normal.

As before, we use $k(X)$ to denote the function field of the underlying
surface $X$, $K$ to denote the algebra of rational functions of a one-dimensional
stratum with nontrivial stabilizer,
and $k^n$ to denote the algebra of functions of a zero-dimensional
orbit of length $n$.
When we blow up such an orbit, we use $k^n(t)$ to denote the total ring
of fractions of the exceptional locus.

\

\noindent
{\bf Generators:}

\noindent
$(C_6,\langle\bar\sigma\rangle \actsfromleft K,(1))$

\noindent
$(C_6,\langle\bar\sigma\rangle \actsfromleft k^2,(1,j))$, \quad $j=1,\ldots, 4$,  \quad

\noindent
$(C_6,\langle\bar\sigma\rangle \actsfromleft k^2,(2,3))$

\noindent
$(D_2,\mathrm{triv}\actsfromleft k,(a_1,a_2))$, \quad $a_1,a_2\in \bF_2^2$, generating $\bF_2^2$

\noindent
$(C_3,\langle \bar\rho,\bar\sigma\rangle\actsfromleft K,(1))$

\noindent
$(C_3,\langle \bar\rho,\bar\sigma\rangle\actsfromleft k^4,(1,1))$

\noindent
$(C_2,\langle \bar\rho,\bar\sigma\rangle\actsfromleft K,(1))$

\noindent
$(S,C_2\actsfromleft K,(1))$

\noindent
$(S',C_2\actsfromleft K,(1))$

\noindent
$(\mathrm{triv},D_6\actsfromleft k(X),\mathrm{triv})$

\

\noindent
{\bf Relations:}

\noindent
$(C_6,\langle\bar\sigma\rangle \actsfromleft k^2,(1,1))=(C_6,\langle \bar\sigma\rangle \actsfromleft k^2(t),(1))$

\noindent
$(C_6,\langle\bar\sigma\rangle \actsfromleft k^2,(1,2))=(C_6,\langle\bar\sigma\rangle \actsfromleft k^2,(1,1))+(C_6,\langle\bar\sigma\rangle \actsfromleft k^2,(1,4))$

\noindent
\begin{multline*}
\hskip -0.35cm
(C_6,\langle\bar\sigma\rangle \actsfromleft k^2,(1,3)) =(C_6,\langle\bar\sigma\rangle \actsfromleft k^2,(1,2))+(C_6,\langle\bar\sigma\rangle \actsfromleft k^2,(2,3))+ \\
(C_2,\langle \bar\rho,\bar\sigma\rangle \actsfromleft k^2(t),(1)),
\end{multline*}
 where $\bar\rho$ acts by cube roots of unity on $t$

\noindent
\begin{multline*}
\hskip -0.35cm
(C_6,\langle\bar\sigma\rangle \actsfromleft k^2,(1,4))=(C_6,\langle\bar\sigma\rangle \actsfromleft k^2,(1,3))+(C_6,\langle\bar\sigma\rangle \actsfromleft k^2,(2,3))+ \\
(C_3,\langle \bar\rho,\bar\sigma\rangle \actsfromleft k^2(t),(1)), 
\end{multline*}
 where $\bar\rho$ acts by $-1$ on $t$
 
 \

\noindent
$0=2(C_6,\langle\bar\sigma\rangle \actsfromleft k^2,(1,4))+(C_2,\langle \bar\rho,\bar\sigma\rangle \actsfromleft k^2(t),(1))$, \\
where $\bar\rho$ acts by cube roots of unity on $t$

\

\noindent
$(C_6,\langle\bar\sigma\rangle \actsfromleft k^2,(2,3))=(C_6,\langle\bar\sigma\rangle \actsfromleft k^2,(1,2))+(C_6,\langle\bar\sigma\rangle \actsfromleft k^2,(1,3))$

\

\noindent
\begin{multline*}
\hskip -0.35cm
(D_2,\mathrm{triv}\actsfromleft k,((1,0),(0,1)))= 
(D_2,\mathrm{triv}\actsfromleft k,((1,0),(1,1)))+ \\ (D_2,\mathrm{triv}\actsfromleft k,((0,1),(1,1)))+ (S',C_2\actsfromleft k(t),(1))
\end{multline*}

\noindent
\begin{multline*}
\hskip -0.35cm
(D_2,\mathrm{triv}\actsfromleft k,((1,0),(1,1)))=(D_2,\mathrm{triv}\actsfromleft k,((1,0),(0,1)))+ \\(D_2,\mathrm{triv}\actsfromleft k,((0,1),(1,1)))+(C_2,\langle \bar\rho,\bar\sigma\rangle \actsfromleft k^3(t),(1)), 
\end{multline*}
 permutation action on $k^3$ with $\bar\sigma$ acting by $-1$ on $t$

\

\noindent
\begin{multline*}
\hskip -0.35cm
(D_2,\mathrm{triv}\actsfromleft k,((0,1),(1,1)))=(D_2,\mathrm{triv}\actsfromleft k,((1,0),(0,1)))+ \\(D_2,\mathrm{triv}\actsfromleft k,((1,0),(1,1)))+(S,C_2\actsfromleft k(t),(1))
\end{multline*}

\

\noindent
$(C_3,\langle \bar\rho,\bar\sigma\rangle\actsfromleft k^4,(1,1))=(C_3,\langle \bar\rho,\bar\sigma\rangle\actsfromleft k^4(t),(1))$

\

\noindent
$0=2(C_3,\langle \bar\rho,\bar\sigma\rangle\actsfromleft k^4,(1,1))$

\noindent
$0=(C_2,\langle \bar\rho,\bar\sigma\rangle\actsfromleft k^6(t),(1))$

\

\noindent
$0=(S,C_2\actsfromleft k^2(t),(1))$

\noindent
$0=(S',C_2\actsfromleft k^2(t),(1))$

\subsection{Embeddings of $\fS_3\times C_2$ into the Cremona group}
\label{subsect:isk}

Iskovskikh \cite{isk-s3} exhibited two nonconjugate
copies of $G=\fS_3\times C_2\cong D_6$
in $\BirAut(\bP^2)$: 
\begin{itemize}
\item{the action on $x_1+x_2+x_3=0$ by permutation and reversing signs,
with model $\bP^2$;}
\item{the action on $y_1y_2y_3=1$ by permutation and taking inverses, 
with model a sextic del Pezzo surface.}
\end{itemize}
To justify the interest in these particular actions we observe that $G$ is the Weyl group of 
the exceptional Lie group $\mathsf G_2$, 
which acts on the Lie algebra of the torus, respectively on the torus itself, and it is natural to ask whether or not 
these actions are equivariantly birational.  
It turns out that they are stably $G$-birational \cite[Proposition 9.11]{lemire}, but not $G$-birational. 
The proof of failure of $G$-birationality in \cite{isk-s3} relies on the classification of links, via the $G$-equivariant Sarkisov program. 

Here we explain how to apply  $\Burn_2(G)$ to this problem. 
Note that neither model above satisfies the stabilizer condition required in the Definition \eqref{eqn:defn-xg}! 
We need to replace the surfaces by appropriate models $X$ and $Y$, in particular, 
to blow up points:
\begin{itemize}
\item{$(x_1,x_2,x_3)=(0,0,0)$, with $G$ as stabilizer;}
\item{$(y_1,y_2,y_3)=(1,1,1)$, with $G$ as stabilizer, and 
$$
(\omega,\omega,\omega), \quad (\omega^2,\omega^2,\omega^2), \quad 
\omega=e^{2\pi i/3},
$$ 
with $\fS_3$ as stabilizer.}
\end{itemize}

We describe these actions in more detail, following closely \cite{isk-s3}. 
The action on $\bP^2$, with coordinates $(u_0:u_1:u_2)$ is given by 
$$
\begin{pmatrix}
1 &  0 &  0  \\ 0 &  0 & 1 \\ 0 &  1 & 0 
\end{pmatrix}, 
\quad 
\begin{pmatrix}
1 & 0 &  0 \\ 0 &  0 & 1 \\ 0 &  -1  & -1 
\end{pmatrix},
\quad 
\rho^3=
\begin{pmatrix}
1 & 0 & 0 \\ 0 &  -1 &  0 \\ 0 &  0 & -1 
\end{pmatrix}.
$$
There is one fixed point, $(1:0:0)$; after blowing up this point, the exceptional curve is stabilized by the 
central involution $\rho^3$, and comes with a nontrivial $\fS_3$-action, contributing the symbol
\begin{equation}
\label{eqn:symbs3}
(C_2, \fS_3 \actsfromleft k(\bP^1), (1)) 
\end{equation}
to $[X\actsfromright G]$. 
Additionally, the line $\ell_0:=\{ u_0=0\}$ has as stabilizer the central $C_2$, contributing the same symbol.  
There are also other contributing terms, of the shape:
\begin{itemize}
\item $(C_6, C_2 \actsfromleft k^2, \beta)$
\item $(D_2, {\rm triv}\actsfromleft k, \beta')$
\end{itemize}
for some weights $\beta, \beta'$. 

A better model for the second action is the quadric
$$
v_0v_1+v_1v_2+v_2v_0=3w^2,
$$
where $\fS_3$ permutes the coordinates $(v_0:v_1:v_2)$ 
and the central involution exchanges the sign on $w$. 
There are no $G$-fixed points, but a conic $R_0:=\{ w=0\}$ with stabilizer the central $C_2$ and 
a nontrivial action of $\fS_3$. There are also: 
\begin{itemize}
\item a $G$-orbit of length 2: 
$$
\{ P_1:=(1:1:1:1), P_2:=(1:1:1:-1)\},
$$ 
exchanged by the central involution, each point has stabilizer $\fS_3$ -- these points have to be blown up, yielding a pair of conjugated $\bP^1$, with a nontrivial $\fS_3$-action;
\item another curve $R_1:=\{ v_0+v_1+v_2=0\}$ with effective $G$-action; 
\item additional points with stabilizers $C_6$ and $D_2$ in $R_0$ and $R_1$. 
\end{itemize}

The essential difference is that the symbol \eqref{eqn:symbs3} appears {\em twice} for the action on $\bP^2$, and only {\em once} for the action on the quadric: the pair of conjugated $\bP^1$ with $\fS_3$-action has trivial stabilizer and does not contribute. Further blow-ups will not introduce new curves of this type.
Formally, examining the relations in Section~\ref{subsect:relations}, we see that the symbol \eqref{eqn:symbs3}
is not equivalent to any combination of other symbols, i.e., it is independent of all other symbols. 
This implies that 
$$
[X\actsfromright G] \neq [Y\actsfromright G] \in \Burn_2(G),
$$
thus $X$ and $Y$ are not $G$-equivariantly birational.
Note, that $X$ and $Y$ are equivariantly birational for any proper subgroup of $G$. 

\begin{rema}
One can view 
the symbol \eqref{eqn:symbs3} 
as the analog of a curve of higher genus in the fixed locus of an element in the classification 
of abelian actions on surfaces, as discussed in Section~\ref{sect:surf}. 
\end{rema}

\bibliographystyle{plain}
\bibliography{symbolssurvey}

\begin{thebibliography}{10}

\bibitem{AdTh}
N.~Addington and R.~Thomas.
\newblock Hodge theory and derived categories of cubic fourfolds.
\newblock {\em Duke Math. J.}, 163(10):1885--1927, 2014.

\bibitem{AtiBot}
M.~F. Atiyah and R.~Bott.
\newblock A {L}efschetz fixed point formula for elliptic complexes. {II}.
  {A}pplications.
\newblock {\em Ann. of Math. (2)}, 88:451--491, 1968.

\bibitem{BLMNPS}
A.~Bayer, M.~Lahoz, E.~Macr\`i, H.~Nuer, A.~Perry, and P.~Stellari.
\newblock Stability conditions in families, 2019.
\newblock \texttt{arXiv:1902.08184}.

\bibitem{baylebeauville}
L.~Bayle and A.~Beauville.
\newblock Birational involutions of {${\bf P}^2$}.
\newblock {\em Asian J. Math.}, 4(1):11--17, 2000.
\newblock Kodaira's issue.

\bibitem{BeaBla}
A.~Beauville and J.~Blanc.
\newblock On {C}remona transformations of prime order.
\newblock {\em C. R. Math. Acad. Sci. Paris}, 339(4):257--259, 2004.

\bibitem{bergh}
D.~Bergh.
\newblock Functorial destackification of tame stacks with abelian stabilisers.
\newblock {\em Compos. Math.}, 153(6):1257--1315, 2017.

\bibitem{blancthesis}
J.~Blanc.
\newblock {\em Finite abelian subgroups of the {C}remona group of the plane}.
\newblock PhD thesis, Universit\'e de Gen\`eve, 2006.
\newblock Th\`ese no. 3777, {\tt arXiv:0610368}.

\bibitem{BlaGGD}
J.~Blanc.
\newblock Linearisation of finite abelian subgroups of the {C}remona group of
  the plane.
\newblock {\em Groups Geom. Dyn.}, 3(2):215--266, 2009.

\bibitem{blancsubgroups}
J.~Blanc.
\newblock Elements and cyclic subgroups of finite order of the {C}remona group.
\newblock {\em Comment. Math. Helv.}, 86(2):469--497, 2011.

\bibitem{BogPro}
F.~Bogomolov and Yu. Prokhorov.
\newblock On stable conjugacy of finite subgroups of the plane {C}remona group,
  {I}.
\newblock {\em Cent. Eur. J. Math.}, 11(12):2099--2105, 2013.

\bibitem{CheShr}
I.~Cheltsov and C.~Shramov.
\newblock Three embeddings of the {K}lein simple group into the {C}remona group
  of rank three.
\newblock {\em Transform. Groups}, 17(2):303--350, 2012.

\bibitem{deFer}
T.~de~Fernex.
\newblock On planar {C}remona maps of prime order.
\newblock {\em Nagoya Math. J.}, 174:1--28, 2004.

\bibitem{DolIsk}
I.~V. Dolgachev and V.~A. Iskovskikh.
\newblock Finite subgroups of the plane {C}remona group.
\newblock In {\em Algebra, arithmetic, and geometry: in honor of {Y}u. {I}.
  {M}anin. {V}ol. {I}}, volume 269 of {\em Progr. Math.}, pages 443--548.
  Birkh\"{a}user Boston, Boston, MA, 2009.

\bibitem{Has00}
B.~Hassett.
\newblock Special cubic fourfolds.
\newblock {\em Compositio Math.}, 120(1):1--23, 2000.

\bibitem{HT17}
B.~Hassett and Yu. Tschinkel.
\newblock Rational points on {K}3 surfaces and derived equivalence.
\newblock In {\em Brauer groups and obstruction problems}, volume 320 of {\em
  Progr. Math.}, pages 87--113. Birkh\"{a}user/Springer, Cham, 2017.

\bibitem{Hau}
O.~Haution.
\newblock Fixed point theorems involving numerical invariants.
\newblock {\em Compos. Math.}, 155(2):260--288, 2019.

\bibitem{IskMin}
V.~A. Iskovskih.
\newblock Minimal models of rational surfaces over arbitrary fields.
\newblock {\em Izv. Akad. Nauk SSSR Ser. Mat.}, 43(1):19--43, 237, 1979.

\bibitem{isk-s3}
V.~A. Iskovskikh.
\newblock Two non-conjugate embeddings of {$S_3\times Z_2$} into the {C}remona
  group. {II}.
\newblock In {\em Algebraic geometry in {E}ast {A}sia---{H}anoi 2005},
  volume~50 of {\em Adv. Stud. Pure Math.}, pages 251--267. Math. Soc. Japan,
  Tokyo, 2008.

\bibitem{kontsevichpestuntschinkel}
M.~Kontsevich, V.~Pestun, and Yu. Tschinkel.
\newblock Equivariant birational geometry and modular symbols, 2019.
\newblock {\tt arXiv:1902.09894}, to appear in J. Eur. Math. Soc.

\bibitem{KT}
M.~Kontsevich and Yu. Tschinkel.
\newblock Specialization of birational types.
\newblock {\em Invent. Math.}, 217(2):415--432, 2019.

\bibitem{Bbar}
A.~Kresch and Yu. Tschinkel.
\newblock Birational types of algebraic orbifolds, 2019.
\newblock {\tt arXiv:1910.07922}.

\bibitem{kreschtschinkel}
A.~Kresch and Yu. Tschinkel.
\newblock Equivariant birational types and {B}urnside volume, 2020.
\newblock {\tt arXiv:2007.12538}.

\bibitem{Kuz}
A.~Kuznetsov.
\newblock Derived categories of cubic fourfolds.
\newblock In {\em Cohomological and geometric approaches to rationality
  problems}, volume 282 of {\em Progr. Math.}, pages 219--243. Birkh\"auser
  Boston, Inc., Boston, MA, 2010.

\bibitem{lemire}
N.~Lemire, V.~L. Popov, and Z.~Reichstein.
\newblock Cayley groups.
\newblock {\em J. Amer. Math. Soc.}, 19(4):921--967, 2006.

\bibitem{LooijengaSurvey}
E.~Looijenga.
\newblock Motivic measures.
\newblock Number 276, pages 267--297. 2002.
\newblock S\'{e}minaire Bourbaki, Vol. 1999/2000.

\bibitem{manin1}
Ju.~I. Manin.
\newblock Rational surfaces over perfect fields.
\newblock {\em Inst. Hautes \'{E}tudes Sci. Publ. Math.}, (30):55--113, 1966.

\bibitem{manin2}
Ju.~I. Manin.
\newblock Rational surfaces over perfect fields. {II}.
\newblock {\em Mat. Sb. (N.S.)}, 72 (114):161--192, 1967.

\bibitem{may}
E.~Mayanskiy.
\newblock Abelian automorphism groups of cubic fourfolds, 2013.
\newblock {\tt arXiv:1308.5150}.

\bibitem{NicSh}
J.~Nicaise and E.~Shinder.
\newblock The motivic nearby fiber and degeneration of stable rationality.
\newblock {\em Invent. Math.}, 217(2):377--413, 2019.

\bibitem{Orlov}
D.~O. Orlov.
\newblock Equivalences of derived categories and {$K3$} surfaces.
\newblock {\em J. Math. Sci. (New York)}, 84(5):1361--1381, 1997.
\newblock Algebraic geometry, 7.

\bibitem{Ouchi}
G.~Ouchi.
\newblock Automorphism groups of cubic fourfolds and {K}3 categories, 2019.
\newblock {\tt arXiv:1909.11033} to appear in {\em Algebraic Geometry}.

\bibitem{ProSimple}
Yu. Prokhorov.
\newblock Simple finite subgroups of the {C}remona group of rank 3.
\newblock {\em J. Algebraic Geom.}, 21(3):563--600, 2012.

\bibitem{prokhorovII}
Yu. Prokhorov.
\newblock On stable conjugacy of finite subgroups of the plane {C}remona group,
  {II}.
\newblock {\em Michigan Math. J.}, 64(2):293--318, 2015.

\bibitem{PSJordan}
Yu. Prokhorov and C.~Shramov.
\newblock Jordan property for {C}remona groups.
\newblock {\em Amer. J. Math.}, 138(2):403--418, 2016.

\bibitem{reichsteinyoussinessential}
Z.~Reichstein and B.~Youssin.
\newblock Essential dimensions of algebraic groups and a resolution theorem for
  {$G$}-varieties.
\newblock {\em Canad. J. Math.}, 52(5):1018--1056, 2000.
\newblock With an appendix by J\'{a}nos Koll\'{a}r and Endre Szab\'{o}.

\bibitem{reichsteinyoussininvariant}
Z.~Reichstein and B.~Youssin.
\newblock A birational invariant for algebraic group actions.
\newblock {\em Pacific J. Math.}, 204(1):223--246, 2002.

\bibitem{sagemath}
The {Sage Developers}.
\newblock {\em {S}ageMath, the {S}age {M}athematics {S}oftware {S}ystem
  ({V}ersion 9.0)}, 2020.
\newblock {\tt https://www.sagemath.org}.

\bibitem{stacks-project}
The {Stacks Project Authors}.
\newblock \textit{Stacks Project}.
\newblock \url{https://stacks.math.columbia.edu}, 2020.

\end{thebibliography}
\end{document}